\documentclass[a4paper,12pt]{article}
\usepackage[utf8]{inputenc}
\usepackage{amsmath,amsthm,amsfonts,latexsym,amssymb,bm,enumerate}

\usepackage[dvips]{graphicx}
\usepackage{psfrag}
\DeclareGraphicsExtensions{.eps,.art,.ART,.ps}

\DeclareFontFamily{OT1}{pzc}{}
\DeclareFontShape{OT1}{pzc}{m}{it}%
              {<-> s * pzcmi8t}{}
\DeclareMathAlphabet{\mathpzc}{OT1}{pzc}%
                                {m}{it}

\newtheorem{Thm}{Theorem}[section]

\newtheorem{Lem}[Thm]{Lemma}

\newtheorem{Prop}[Thm]{Proposition}
\newtheorem{Claim}[Thm]{Claim}

\newtheorem{Rmk}[Thm]{Remark}

\newcommand{\R}{\mathbb{R}}

\newcommand{\eps}{\varepsilon}
\newcommand{\weak}{\rightharpoonup}
\newcommand{\h}{{\cal H}}
\newcommand{\C}{{\cal C}}
\newcommand{\cN}{{\cal N}}

\newcommand{\D}{{\cal D}}
\newcommand{\cS}{{\cal I}_\phi}
\newcommand{\cSt}{{\cal I}_\psi}
\newcommand{\lin}{{\cal A}}
\newcommand{\rr}{{\cal R}}
\newcommand{\rs}{{\cal S}}
\newcommand{\tpy}{(t\phi+y)}
\newcommand{\tpsy}{(t\psi+y)}

\newcommand{\twzy}{(tw_*+y)}
\newcommand{\twyz}{(t_*w_*+y_*)}
\newcommand{\A}{a}
\newcommand{\uu}{u}

\newcommand{\y}{y}
\newcommand{\ttt}{t_\dagger}
\newcommand{\cc}{c}
\newcommand{\rrro}{\mathpzc{r}_{\,1}}
\newcommand{\rrrt}{\mathpzc{r}_{\,2}}
\newcommand{\rrr}{\mathpzc{r}}
\newcommand{\w}{w}
\newcommand{\tu}{\tilde{u}}

\newcommand{\tH}{\tilde{H}}
\newcommand{\tG}{\tilde{G}}
\newcommand{\tg}{\tilde{g}}
\newcommand{\tua}{\tilde{u}(a)}

\newcommand{\ufs}{u_{\varsigma}}
\newcommand{\afs}{a_{\varsigma}}
\newcommand{\cfs}{c_{\varsigma}}
\newcommand{\cfsup}{c^{\flat\sharp}}
\newcommand{\yfs}{y_{\varsigma}}
\newcommand{\wfs}{\zeta_{\varsigma}}
\renewcommand{\tau}{t_*}

\begin{document}
\begin{center}
{\Large\bf 
Bifurcation curves 
of a 
 logistic equation
when the linear growth rate crosses
a second eigenvalue}\\
\ \\
Pedro Martins Girão\footnote{Email: pgirao@math.ist.utl.pt. 
Partially supported by the Fundação para a Ciência e a Tecnologia (Portugal) and by project
UTAustin/MAT/0035/2008.
} 
\\

\vspace{2.2mm}

{\small Instituto Superior Técnico, Av.\ Rovisco Pais,
1049-001 Lisbon, Portugal}

\end{center}

\begin{center}
{\bf Abstract}
\end{center}
We construct the global bifurcation curves, solutions versus level of harvesting,
for the steady states of a diffusive logistic equation on a bounded domain,
under Dirichlet boundary conditions and other appropriate hypotheses,
when $a$, the linear growth rate of the population, is below $\lambda_2+\delta$. 
Here $\lambda_2$ is the second eigenvalue of the Dirichlet Laplacian on the domain and $\delta>0$.
Such curves have been obtained before, but only for $a$ in a right neighborhood of the first eigenvalue. 
Our analysis provides the exact number of solutions of the equation
for $a\leq\lambda_2$ and new information on the number of solutions for $a>\lambda_2$.

\noindent

\section{Introduction}

Diffusive logistic equations with harvesting 
are equations of the form
\begin{equation}\label{a}
-\Delta u=au-f(u)-ch
\end{equation}
in a domain $\Omega$, with some boundary conditions, here taken to be homogeneous Dirichlet,
for a competition term $f$ and a harvesting function $h$ and level $c$.
In \cite{CDT}, \cite{Eu} and the references therein the reader may find recent work regarding this subject.

In \cite{OSS1} the authors obtained global
bifurcation diagrams for solutions of~(\ref{a})
when the competition term
is proportional to the square of the population~$u$.
The original motivation for our study was  
somewhat limited. 
We wanted to be able to deal with
other competition terms whose
second derivative vanishes at zero.
Building on the work in~\cite{OSS1}, we achieve our original goal and, more importantly, prove several new results.
We denote by $\lambda_1$ and $\lambda_2$ the first and second eigenvalues of the Dirichlet Laplacian
on $\Omega$, respectively.
Whereas in~\cite{OSS1} the authors, who were seeking positive solutions,
obtained global bifurcation curves for
the parameter $a$ below $\lambda_1+\delta$, under suitable additional hypotheses we 
obtain global bifurcation curves for any $a$ below $\lambda_2+\delta$.
Part of the solutions on these curves will change sign.

We should mention that this paper is also related to problems addressing
the so called jumping nonlinearities. Here, unlike the most common hypothesis
(see, for example, \cite{CHV}), the nonlinearity is not asymptotically linear on
one of the ends of the real line.

Let us now state our most important results.
Let $\Omega$ be a smooth bounded domain in $\R^N$ with $N\geq 1$, $p>N$
and $\h=\{u\in W^{2,p}(\Omega): u=0\ {\rm on}\ \partial\Omega\}$.
We are interested in weak solutions of the equation (\ref{a})
belonging to the space $\h$. 
In the first sections of the paper we just assume 
\begin{enumerate}[{\bf (a)}]
 \item[{\bf (i)}] $f\in C^2(\R)$.
 \item[{\bf (ii)}] $f(u)=0$ for $u\leq M$, and $f(u)>0$ for $u>M$; throughout $M\geq 0$ is fixed.
 \item[{\bf (iii)}] $f''(u)\geq 0$. 
 \item[{\bf (iv)}] $\lim_{u\to+\infty} \frac{f(u)}{u}=+\infty$.
\item[{\bf (a)}] $h\in L^\infty(\Omega)$.
\item[{\bf (b)}] $h\geq 0$ in $\Omega$ and $h>0$ on a set of positive measure.
\end{enumerate}
For $a$ below $\lambda_2$ we prove
\begin{Thm}
\label{l1l2}
Suppose $f$ satisfies {\rm\bf (i)}-{\rm\bf (iv)}
and $h$ satisfies {\rm\bf (a)}-{\rm\bf (b)}.
 Fix $\lambda_1<a<\lambda_2$.
 The set of solutions $(c,u)$ of {\rm (\ref{a})} is a connected one dimensional manifold ${\cal M}$ of class $C^1$
in $\R\times\h$. We have 
$${\cal M}={\cal M}^\sharp\cup\{{\bm p}_*\}\cup{\cal M}^*,$$
where $\{{\bm p}_*\}$ connects
${\cal M}^\sharp$ and ${\cal M}^*$.
Here
\begin{itemize}
\item
${\cal M}^\sharp$ is the manifold of nondegenerate solutions with Morse index equal to one, which
is a graph $\{(c,u^\sharp(c)):c\in\,]-\infty,c_*[\}$.
\item ${\bm p}_*=(c_*,u_*)$ is a degenerate solution with Morse index equal to zero.
\item ${\cal M}^*$ is the manifold of stable solutions, which
is a graph $\{(c,u^*(c)):c\in\,]-\infty,c_*[\}$.
\end{itemize}
\end{Thm}
Theorem~\ref{l1l2} is illustrated in Figure~\ref{fig1}.
\begin{figure}
\centering
\begin{psfrags}
\psfrag{c}{{\tiny $c$}}
\psfrag{u}{{\tiny $u$}}
\psfrag{a}{{\tiny $(c,u^*(c))$}}
\psfrag{f}{{\tiny $(0,t\psi)$}}
\psfrag{N}{{\tiny $\!\!\!\!\!\!-\frac{M}{\beta}\psi$}}
\psfrag{b}{{\tiny $(c,u^\sharp(c))$}}
\psfrag{e}{{\tiny $(c,u^\flat(c))$}}
\psfrag{d}{{\tiny $(c_*,u_*)$}}
\includegraphics[scale=.55]{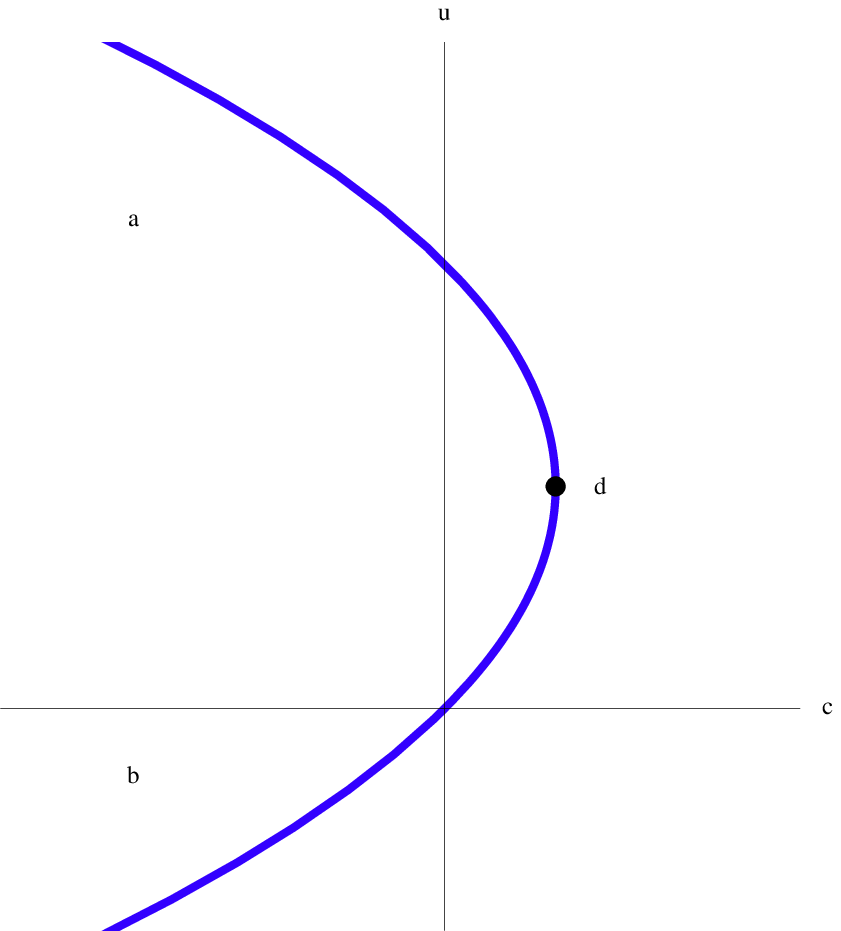}\ \ \ \ \ \ \ \ \ \ \ \ 
\includegraphics[scale=.54]{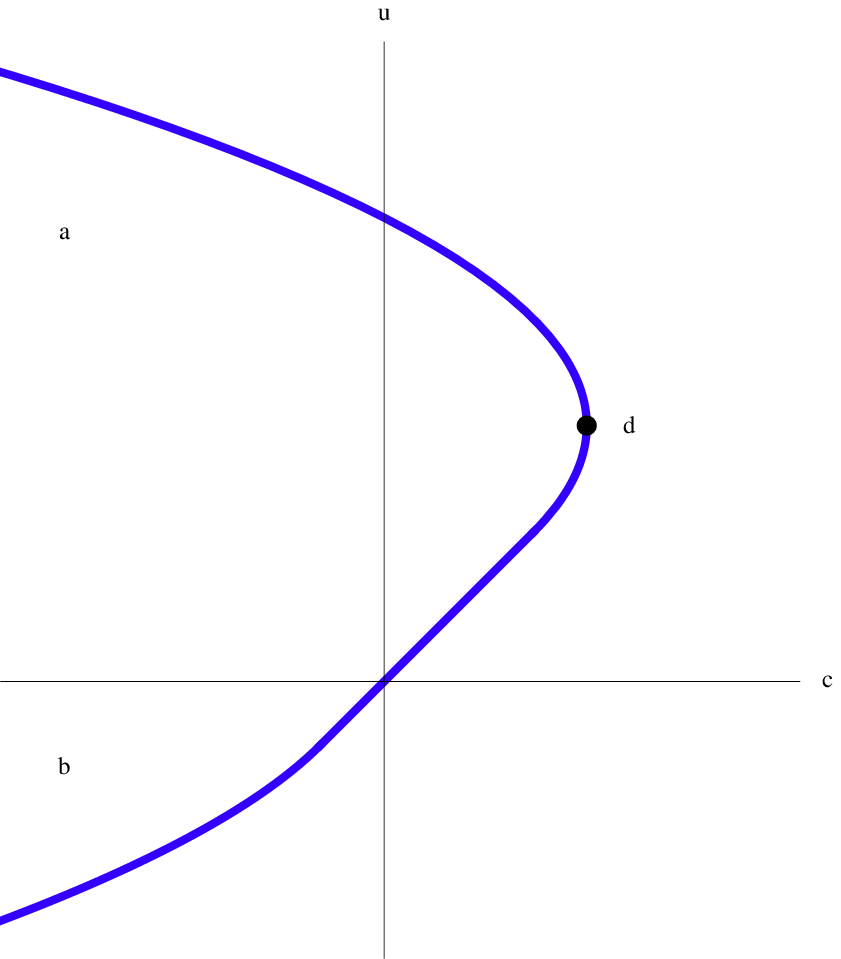}
\end{psfrags}
\caption{Bifurcation curve for $\lambda_1<a<\lambda_2$. 
On the left $M=0$ and on the right $M>0$.}\label{fig1}
\end{figure}

To study (\ref{a}) for $a\geq\lambda_2$ we make additional assumptions.
Specifically, we assume 
\begin{enumerate}
 \item[{\bf ($\bm\alpha$)}] $\lambda_2$ is simple, with eigenspace spanned by $\psi$.
 \item[{\bf (c)}] $\int h\psi\neq 0$. 
\end{enumerate}
When the region of integration is omitted it is understood to be $\Omega$.
We denote by $\phi$ the first eigenfunction of the Dirichlet Laplacian
satisfying $\max_\Omega\phi=1$,
and
we also normalize the second eigenfunction $\psi$ 
so $\max_\Omega\psi=1$. We define 
\begin{equation}\label{beta}
\beta=-\min_\Omega\psi,
\end{equation}
 so that $\beta>0$. 
For $a$ equal to $\lambda_2$ we prove
\begin{Thm}
\label{thmatl2} Suppose $f$ satisfies {\rm\bf (i)}-{\rm\bf (iv)}, {\rm $\bm(\bm\alpha\bm)$} holds
and $h$ satisfies {\rm\bf (a)}-{\rm\bf (c)}.
 Fix $a=\lambda_2$. 
 The set of solutions $(c,u)$ of {\rm (\ref{a})} is a connected one dimensional manifold ${\cal M}$ of class $C^1$
in $\R\times\h$. We have 
$${\cal M}={\cal M}^\flat\cup{\cal L}\cup{\cal M}^\sharp\cup\{{\bm p}_*\}\cup{\cal M}^*,$$
where ${\cal L}$ connects ${\cal M}^\flat$ and ${\cal M}^\sharp$, and $\{{\bm p}_*\}$ connects
${\cal M}^\sharp$ and ${\cal M}^*$.
Here
\begin{itemize}
\item
${\cal M}^\flat$ is a manifold of nondegenerate solutions with Morse index equal to one, which
is a graph $\{(c,u^\flat(c)):c\in\,]-\infty,0[\}$.
\item
${\cal L}$ is a segment (a point in the case $M=0$) of degenerate solutions with Morse index
equal to one, $\bigl\{(0,t\psi):t\in\bigl[-\frac{M}{\beta},M\bigr]\bigr\}$. 
\item
${\cal M}^\sharp$ is a manifold of nondegenerate solutions with Morse index equal to one, which
is a graph $\{(c,u^\sharp(c)):c\in\,]0,c_*[\}$.
\item ${\bm p}_*=(c_*,u_*)$ is a degenerate solution with Morse index equal to zero.
\item ${\cal M}^*$ is the manifold of stable solutions, which
is a graph $\{(c,u^*(c)):c\in\,]-\infty,c_*[\}$.
\end{itemize}
\end{Thm}
Theorem~\ref{thmatl2} is illustrated in Figure~\ref{fig2}.
\begin{figure}
\centering
\begin{psfrags}
\psfrag{c}{{\tiny $c$}}
\psfrag{u}{{\tiny $u$}}
\psfrag{a}{{\tiny $(c,u^*(c))$}}
\psfrag{f}{{\tiny $(0,t\psi)$}}
\psfrag{N}{{\tiny $\!\!\!\!\!\!-\frac{M}{\beta}\psi$}}
\psfrag{b}{{\tiny $(c,u^\sharp(c))$}}
\psfrag{e}{{\tiny $(c,u^\flat(c))$}}
\psfrag{d}{{\tiny $(c_*,u_*)$}}
\includegraphics[scale=.6]{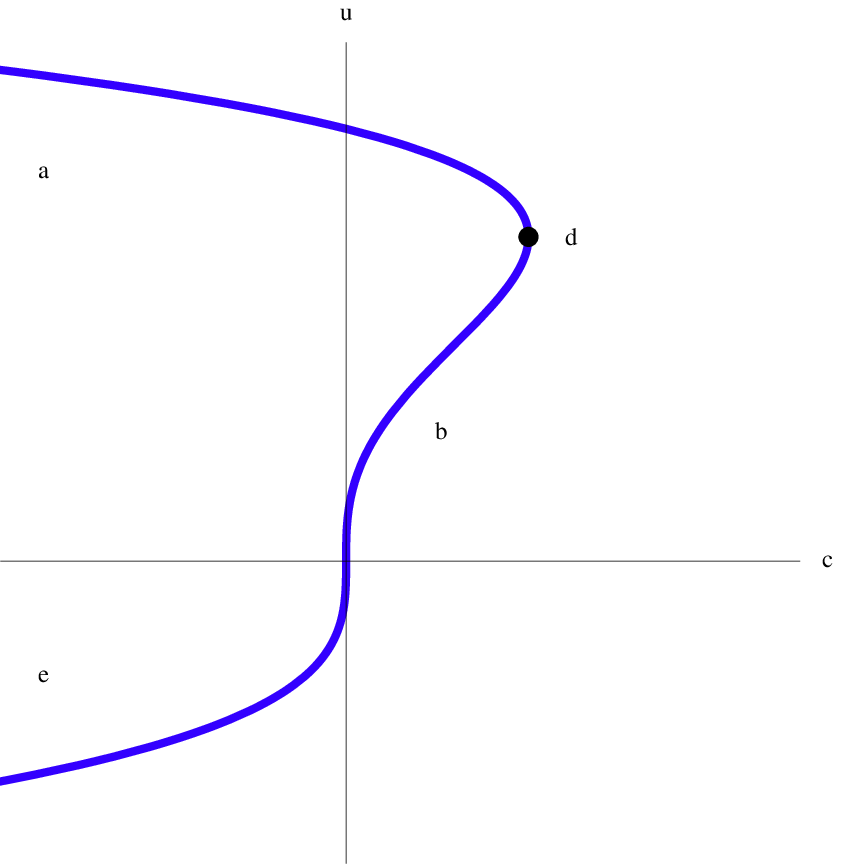}\ \ \ \ \ \ \ \ \ \ \ \ 
\includegraphics[scale=.54]{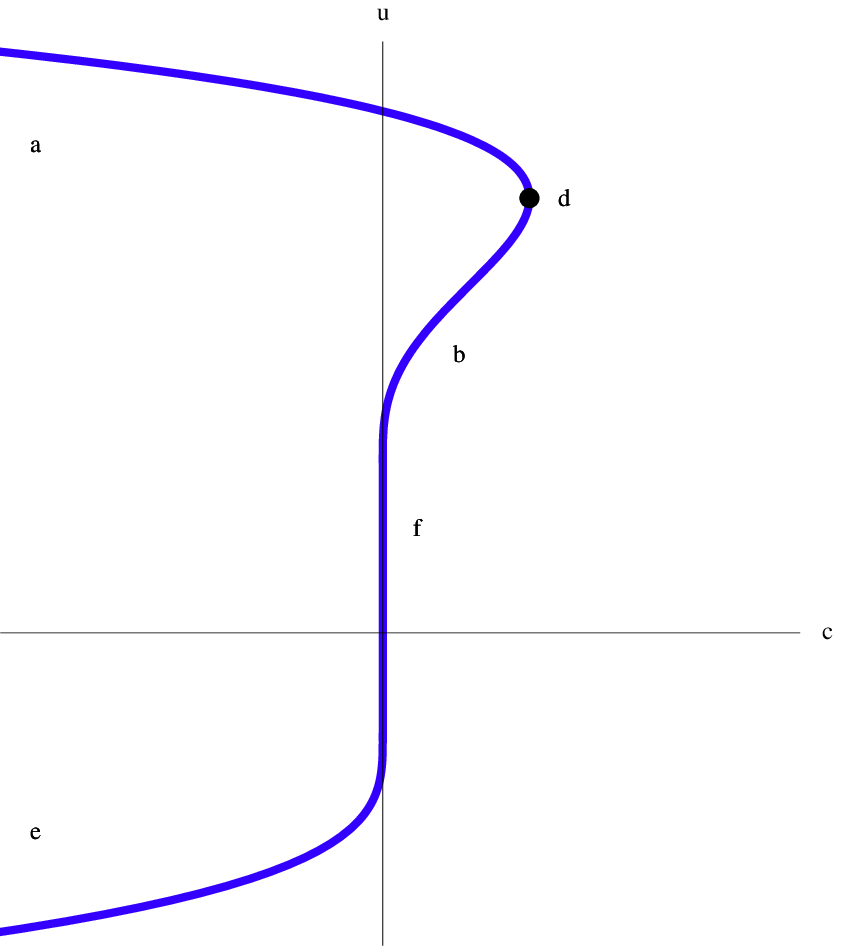}
\end{psfrags}
\caption{Bifurcation curve for $a=\lambda_2$. On the left $M=0$ and on the right $M>0$.}\label{fig2}
\end{figure}
\begin{Rmk}\label{rmk1}
If $M=0$ the set of solutions with Morse index equal to one is 
the graph of a continuous function 
$$u^{\flat\sharp}(c)=\left\{
\begin{array}{ll}
u^\flat(c)&{\rm for}\ c<0,\\
0&{\rm for}\ c=0,\\
u^\sharp(c)&{\rm for}\ 0<c<c_*,
\end{array}
\right.
$$
which is not differentiable at zero.
\end{Rmk}

We can go a little bit beyond $\lambda_2$ provided we strengthen {\bf (b)} to
\begin{enumerate}
 \item[{\bf (b)$'$}] $h>0$ a.e.\ in $\Omega$.
\end{enumerate}
To fix ideas, without loss of generality, suppose
\begin{equation}\label{hpsi}
\int h\psi<0.
\end{equation}
We define
\begin{equation}\label{s}
\textstyle\rs:=\bigl\{y\in\h:\int y\psi=0\bigr\}.
\end{equation}
For $a$ above $\lambda_2$ we prove
\begin{Thm}
\label{thm} Suppose $f$ satisfies {\rm\bf (i)}-{\rm\bf (iv)}, {\rm $\bm(\bm\alpha\bm)$} holds
and $h$ satisfies {\rm\bf (a)}, {\rm\bf (b)$'$}, {\rm\bf (c)}. Without loss of generality, suppose {\rm (\ref{hpsi})} is true.
There exists $\delta>0$ such that the following holds.
Fix $\lambda_2<a<\lambda_2+\delta$. 
The set of solutions $(c,u)$ of~{\rm (\ref{a})} is a connected one dimensional manifold ${\cal M}$ of class $C^1$
in $\R\times\h$. We have ${\cal M}$ is the disjoint union
$${\cal M}={\cal M}^\flat\cup\{{\bm p}_\flat\}\cup{\cal M}^\natural\cup\{{\bm p}_\sharp\}\cup{\cal M}^\sharp\cup\{{\bm p_*}\}\cup{\cal M}^*,$$
where $\{{\bm p}_\flat\}$ connects ${\cal M}^\flat$ and ${\cal M}^\natural$, $\{{\bm p}_\sharp\}$ connects
${\cal M}^\natural$ and ${\cal M}^\sharp$, and $\{{\bm p}_*\}$ connects ${\cal M}^\sharp$ and ${\cal M}^*$.
Here
\begin{itemize}
\item ${\cal M}^\flat$ is a manifold 
of nondegenerate solutions with Morse index equal to one, which
is a graph $\{(c,u^\flat(c)):c\in\,]-\infty,c_\flat[\}$.
\item
${\bm p}_\flat=(c_\flat,u_\flat)$ is a degenerate solution with Morse index equal to one.
\item ${\cal M}^\natural$ is a manifold
of solutions with Morse index equal to one or to two, 
$$\bigl\{(\cc^\natural(t),u^\natural(t)):u^\natural(t)=t\psi+\y^\natural(t),\ t\in J\bigr\},$$
with $\cc^\natural:J\to\R$, 
$\y^\natural:J\to{\cal S}$ and
$J=\bigl]-\frac{M}{\beta}-\eps_\flat,M+\eps_\sharp\bigr[$, for some $\eps_\flat,\eps_\sharp>0$.
\item
${\bm p}_\sharp=(c_\sharp,u_\sharp)$ is a degenerate solution with Morse index equal to one.
\item ${\cal M}^\sharp$ is a manifold 
of nondegenerate solutions with Morse index equal to one, which
is a graph $\{(c,u^\sharp(c)):c\in\,]c_\sharp,c_*[\}$.
\item
${\bm p}_*=(c_*,u_*)$ is a degenerate solution with Morse index equal to zero.
\item ${\cal M}^*$ is the manifold of stable solutions, which
is a graph $\{(c,u^*(c)):c\in\,]-\infty,c_*[\}$.
\end{itemize}
We have $(\cc^\natural)'(0)<0$ and
$$
\lim_{t\to -\frac{M}{\beta}-\eps_\flat}(\cc^\natural(t),u^\natural(t))\ =\ (c_\flat,u_\flat),\ 
\lim_{t\to M+\eps_\sharp}(\cc^\natural(t),u^\natural(t))\ =\ (c_\sharp,u_\sharp).
$$
In particular, if $|c|$ is sufficiently small,
then {\rm (\ref{a})} has at least four solutions.
\end{Thm}
Theorem~\ref{thm} is illustrated in Figure~\ref{fig3}.
\begin{figure}
\centering
\begin{psfrags}
\psfrag{c}{{\tiny $c$}}
\psfrag{u}{{\tiny $u$}}
\psfrag{b}{{\tiny $(c_\flat,u_\flat)$}}
\psfrag{n}{{\tiny $(c_\sharp,u_\sharp)$}}
\psfrag{s}{{\tiny $(c_*,u_*)$}}
\psfrag{h}{{\tiny $(c,u^\sharp(c))$}}
\psfrag{m}{{\tiny $(c^\natural(t),u^\natural(t))$}}
\psfrag{d}{{\tiny $(c,u^\flat(c))$}}
\psfrag{t}{{\tiny $(c,u^*(c))$}}
\includegraphics[scale=.65]{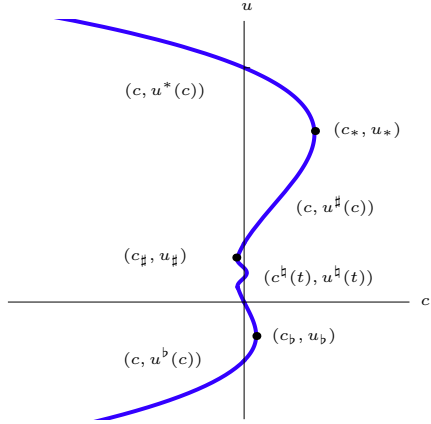}
\end{psfrags}
\caption{A bifurcation curve for $\lambda_2<a<\lambda_2+\delta$.}\label{fig3}
\end{figure}
\begin{Rmk}\label{rmk2}
The sign of $(\cc^\natural)^\prime$ allows us to classify the solutions on ${\cal M}^\natural$. Assuming {\rm (\ref{hpsi})} holds
\begin{itemize}
\item
If $(\cc^\natural)^\prime(t)>0$, then $(\cc^\natural(t),u^\natural(t))$
is a nondegenerate solution with Morse index equal to one.
\item
If $(\cc^\natural)^\prime(t)<0$, then $(\cc^\natural(t),u^\natural(t))$
is a nondegenerate solution with Morse index equal to two.
\item
If $(\cc^\natural)^\prime(t)=0$, then $(\cc^\natural(t),u^\natural(t))$ 
is a degenerate solution with Morse index equal to one.
\end{itemize}
\end{Rmk}

The above three theorems are our main results. 
In addition, we also prove
Theorem~\ref{thmd} (degenerate solutions with Morse index equal to zero),
Proposition~\ref{p1} (behavior of $c$ along the curve of degenerate solutions
with Morse index equal to zero), Proposition~\ref{p2} 
(stable solutions are superharmonic for small $|c|$), 
Theorem~\ref{i1}
(existence of at least three solutions for $a>\lambda_2$, $a$ not an eigenvalue and small $|c|$) and
Proposition~\ref{tp2}
(solutions for the case $c=0$, around and bifurcating from $(\lambda_2,0)$).

The above results are, to our knowledge, new. The following are minor improvements of some
of the theorems in \cite{OSS1}.
Theorem~\ref{thmcz}
(solutions for the case $c=0$, bifurcating from $(\lambda_1,0)$) and
Theorem~\ref{global} (with further information on the case $\lambda_1\leq a<\lambda_1+\delta$)
can be found in \cite{OSS1} (in the case $M=0$).
A big portion of the proof of Theorem~\ref{thmcz} is identical to the one of \cite[Theorem~2.5]{OSS1},
but for the uniqueness argument we do not use 
\cite[Lemma~3.3]{ABC}.
The statement of 
Theorem~\ref{big} (stable solutions of {\rm (\ref{a})}) enriches \cite[Theorem~3.2]{OSS1},
but the argument of the proof can be found in \cite{OSS1}.

Our main tools are bifurcation theory, the Morse indices, critical points at infinity
and the methods of elliptic equations. With regard to the first, the
simplest approach for our purposes is to make the best choice of
coordinates in each circumstance and then apply the Implicit Function
Theorem. This also allows us to simplify some of the original arguments
(namely in what concerns
\cite[Lemma~4.3]{OSS1}).

The organization of this paper is as follows.
In 
Section~\ref{noh} we consider stable
solutions of the equation with no harvesting.
In Section~\ref{sectiond} we consider
degenerate solutions with Morse index equal to zero.
In Section~\ref{four} we consider
stable solutions, solutions around a degenerate solution with Morse index equal to zero,
solutions around zero, and mountain pass solutions.
In Section~\ref{below_l2} we discuss
global bifurcation below $\lambda_2$ and prove Theorem~\ref{l1l2}.
In Section~\ref{at_l2} we discuss
global bifurcation at $\lambda_2$ and prove Theorem~\ref{thmatl2}.
In Section~\ref{pl2} we discuss
bifurcation in a right neighborhood of $\lambda_2$ and 
start the proof of Theorem~\ref{thm}. Finally,
in order to extend the curves obtained in Section~\ref{pl2}
for all negative values of $c$, and complete the proof of
Theorem~\ref{thm}, we need a somewhat delicate argument to prove
there are no degenerate solutions with Morse index equal to one at infinity
for $a<\lambda_2+\delta$. 
In Section~\ref{last} we carry it out.

We finish the Introduction with
a word of caution about our terminology. For simplicity,
we sometimes refer to a solution $(a,u,c)\in\R\times\h\times\R$ of (\ref{a}) in an abbreviated manner,
by $(c,u)$ when $a$ is fixed, or simply by $u$ when both $a$ and $c$ are fixed. And we may also refer to some property of a solution
$(a,u,c)$, like positivity, meaning the second component $u$ has that property.

\vspace{\baselineskip}

{\bf Acknowledgments.} 

The author is grateful to Hossein Tehrani for valuable discussions and helpful suggestions. 

The author studied 
\cite{OSS1} with
José Maria Gomes, whose insight he retains with appreciation.

\section{Stable solutions of the equation with no harvesting}\label{noh}

Throughout this section we assume $c=0$, so we consider the equation
\begin{equation}\label{b}
-\Delta u=au-f(u).
\end{equation}
Here the main result is
\begin{Thm}[$\C_\dagger$, solutions of {\rm (\ref{b})} bifurcating from $(\lambda_1,0)$]\label{thmcz}
Suppose $f$ satisfies {\rm\bf (i)}-{\rm\bf (iv)}.
The set of positive solutions $(a,u)$ of {\rm (\ref{b})} is a connected one dimensional manifold 
$\C_\dagger$ of class $C^1$ in $\R\times\h$.
The manifold is the union of the 
segment $\{(\lambda_1,t\phi):t\in\,]0,M]\}$ with a graph $\{(a,u_\dagger(a)):a\in\,]\lambda_1,+\infty[\}$.
The solutions are strictly increasing along $\C_\dagger$.
For $a>\lambda_1$ every positive solution is stable and, at each $a$, equation {\rm (\ref{b})} has no other stable solution besides $u_\dagger(a)$.
\end{Thm}
\begin{Rmk}
 In {\rm Theorem~\ref{thmcz}} we may relax assumption {\bf (iii)} to
\begin{itemize}
 \item[{\bf (iii)$'$}] $u\mapsto\frac{f(u)}{u}$ is increasing (not necessarily strictly).
\end{itemize}
\end{Rmk}
We define
$$
\textstyle\rr:=\bigl\{y\in\h:\int y\phi=0\bigr\}.
$$
\begin{Lem}[Initial portion of $\C_\dagger$]\label{tp}  Suppose $M>0$.
There exists $\delta>0$ and $C^1$ functions $\A_\dagger:J\to\R$ and $\y_\dagger:J\to\rr$,
where 
$J=]-\infty, 
M+\delta[$, such that 
the map $t\mapsto(\A_\dagger(t),t\phi+\y_\dagger(t))$, defined in $J$, 
with $\A_\dagger(t)=\lambda_1$ and $\y_\dagger(t)= 0$ for $t\in\,]-\infty,
M]$,
parametrizes a curve $\C_\dagger$ of solutions of~{\rm (\ref{b})}.
There exists a neighborhood of $\C_\dagger\setminus\{(\lambda_1,0)\}$
in $\R\times\h$ such that the solutions of~{\rm (\ref{b})} in this neighborhood lie on $\C_\dagger$.
\end{Lem}
\begin{proof}[Sketch of the proof]\footnote{For the full proof see the Appendix~\ref{appendix}.}
 Apply the Implicit Function Theorem to the function $g:\R^2\times\rr\to L^p(\Omega)$,
defined by
$$
g(a,t,y)=\Delta\tpy+a\tpy-f\tpy
$$
at $(\lambda_1,t_0,0)$ with $t_0\in\,]-\infty,0[\,\cup\,]0,M]$.
\end{proof}
Of course the line $\lin\subset\R\times\h$ parametrized by $a\mapsto (a,0)$ is also a curve of solutions
of (\ref{b}). From the classical paper \cite[Theorem~1.7]{CR} we know that $(\lambda_1,0)$ is
a bifurcation point.  We deduce that the statement of
Lemma~\ref{tp} also holds when $M=0$. In both cases, $M>0$ and $M=0$, in a neighborhood of 
$(\lambda_1,0)$ the solutions of~(\ref{b}) lie on $\lin\cup\C_\dagger$.

Let $\delta$ be as in Lemma~\ref{tp} and $u_\dagger(t)=t\phi+\y_\dagger(t)$. Reducing $\delta$ if necessary,
we may assume $u_\dagger(t)>0$ and $\max_\Omega u_\dagger(t)>M$ for $t\in\,]M,M+\delta[$. Indeed,
\begin{equation}\label{texists}
u_\dagger(t)=M\phi+(t-M)\phi+o(|t-M|),
\end{equation} as $y_\dagger'(M)= 0$, and $\phi$ has negative exterior normal 
derivative on $\partial\Omega$.
It follows $$\A_\dagger(t)>\lambda_1\ \ {\rm if}\ t\in\,]M,M+\delta[.$$ To see this, we multiply both
sides of (\ref{b}) by $\phi$ and integrate over $\Omega$,
\begin{equation}\label{al}
(a-\lambda_1)\int u\phi=\int f(u)\phi.
\end{equation}
Notice $f(u_\dagger(t))\not\equiv 0$ because $\max_\Omega u_\dagger(t)>M$.

Let $u$ be a solution of (\ref{b}), $\mu$ be the smallest eigenvalue of the linearized problem at $u$
and $v\in\h$ be a corresponding eigenfunction. So
\begin{equation}\label{linear}
-(\Delta v+av-f'(u)v)=\mu v.
\end{equation}
We recall that either $v$ or $-v$ is strictly positive everywhere in $\Omega$. 
The solution $u$ is said to be 
stable if the smallest eigenvalue of the linearized problem is positive.
The Morse index of the solution $u$ is the number of negative eigenvalues of the 
linearized problem at $u$. The solution $u$ is said to be 
degenerate if one of the eigenvalues of the linearized problem is equal to zero.
Otherwise it is called nondegenerate.
[These definitions also hold for solutions of (\ref{a})].
\begin{Lem}[Necessary and sufficient condition for the stability of solutions of (\ref{b})]\label{gM}
A solution of {\rm (\ref{b})} is stable iff it is a nonnegative solution of {\rm (\ref{b})} with maximum strictly greater than $M$.
\end{Lem}
\begin{proof}
Suppose $u$ is a nonnegative solution of (\ref{b}) whose maximum is strictly greater than $M$.
 If we multiply both sides of (\ref{b}) by $v$, both sides of (\ref{linear}) by $u$,
subtract and integrate over $\Omega$, we obtain
$$
\int(f'(u)u-f(u))v=\mu\int uv.
$$
The hypotheses imply $f'(u)u-f(u)\geq 0$. On the other hand $f'(u)u-f(u)\not\equiv 0$ because
$\max_\Omega u>M$, $f(u)=0$ for $u\leq M$, $f(u)>0$ for $u>M$, $f$ is $C^2$ and $u=0$ on
$\partial\Omega$.
Since $v$ is strictly positive on $\Omega$ and $u$ is nonnegative, $\mu>0$.

Conversely, suppose $u$ is a stable solution of (\ref{b}).
We multiply 
both sides of (\ref{b}) by $u^-$ and integrate over $\Omega$ to obtain
$$
\int|\nabla u^-|^2=a\int (u^-)^2.
$$
Using this equality,
$$
\mu\int (u^-)^2\leq \int|\nabla u^-|^2-a\int (u^-)^2+\int f'(u)(u^-)^2=0,
$$
or $u^-=0$, because $\mu>0$. So $u$ is nonnegative. If $\max_\Omega u\leq M$,
then $u$ satisfies $\Delta u+au=0$. This implies $a=\lambda_1$. But then $u$
is not stable, it is degenerate. Therefore $\max_\Omega u>M$.
\end{proof}
\begin{proof}[Proof of {\rm Theorem~\ref{thmcz}}]
Consider the function $F:\R\times\h\to L^p(\Omega)$,
defined by
$$
F(a,u)=\Delta u+au-f(u).
$$
We know $F(\A_\dagger(t),u_\dagger(t))=0$, in particular for $t\in\,]M,M+\delta[$.
Since for these values of $t$ the functions $u_\dagger(t)$ are positive with
maxima strictly greater than $M$, and hence nondegenerate,
by the Implicit Function Theorem,
in a neighborhood of $(\A_\dagger(t),u_\dagger(t))$, the solutions of $F(a,u)=0$ in $\R\times\h$ may also
be written in the form $(a,u_\dagger(a))$. 
Of course, in rigor $u_\dagger(t)$ should be called something else like
$\tilde{u}_\dagger(t)$, as it is not $u_\dagger(a)$ but rather $\tilde{u}_\dagger(t)=
u_\dagger(a_\dagger(t))$.
However, there is no risk of confusion.

Using the argument in \cite[proof of Theorem~2.5]{OSS1}, one may extend $\C_\dagger$ of Lemma~\ref{tp}
with a curve (still called $\C_\dagger$) of solutions $(a,u_\dagger(a))$
of (\ref{b}) defined for $a\in\,]\lambda_1,+\infty[$. By the maximum principle, $u_\dagger'(a)>0$.
Let $K_a$ satisfy
\begin{equation}\label{ka}
 K_a>0\quad{\rm and}\quad aK_a-f(K_a)=0.
\end{equation} 
Note $au-f(u)\leq 0$ for $u>K_a$ because the hypotheses imply $u\mapsto f(u)/u$
is increasing. It is well known the maximum principle implies $u_\dagger(a)\leq K_a$
and then Hopf's Lemma \cite[Lemma~3.4]{GT} implies $u_\dagger(a)<K_a$.
This will be used in the proof of Proposition~\ref{p2}.

Let
\begin{equation}\label{t}
\ttt(a):=\textstyle\frac{\int u_\dagger(a)\phi}{\int\phi^2}. 
\end{equation}
Clearly, $\lim_{a\searrow\lambda_1}\ttt(a)=M$. We remark that $\ttt$
is strictly increasing along $\C_\dagger$ as 
$$
 \textstyle\ttt'(a)=\frac{\int u_\dagger'(a)\phi}{\int\phi^2}>0.
$$
The next lemma completes the proof of
Theorem~\ref{thmcz}.\end{proof}

\begin{Lem}[Uniqueness of stable solutions of (\ref{b})]\label{unique}
 For each $a\in\,]\lambda_1,+\infty[$, $u_\dagger(a)$ is the unique stable solution of equation~{\rm (\ref{b})}.
\end{Lem}
\begin{proof} 
As mentioned in the Introduction,
the argument below gives a direct
proof of the uniqueness part of Theorem~2.5 of~\cite{OSS1}, which does not use
\cite[Lemma~3.3]{ABC}. The emphasis in \cite{OSS1} is on uniqueness of positive solutions
and here is on uniqueness of stable solutions. Of course, these are the same via Lemma~\ref{gM}.

 Let $\tua$ satisfy $F(a,\tua)=0$ and be stable. 
 We proved in Lemma~\ref{gM} $\tua$ is nonnegative with maximum strictly greater than $M$.
Let us prove
we may use the Implicit Function Theorem to follow the solution $\tua$ as $a$ decreases all the way down to $\lambda_1$.
The solutions $\tua$ will not blow up to $+\infty$ as $\tua<K_a$. And
they will remain bounded below by zero as long as they remain stable.
We now show the solutions $\tua$ will remain
stable for $a>\lambda_1$.  
Suppose $a_n\searrow a_0>\lambda_1$ and $\tu(a_n)$ are stable solutions.
Multiplying both sides of (\ref{b}) by $\tu(a_n)$ and integrating over $\Omega$,
\begin{equation}\label{bdd}
\int|\nabla \tu(a_n)|^2=a_n\int(\tu(a_n))^2-\int f(\tu(a_n))\tu(a_n).
\end{equation}
The sequence $(\tu(a_n))$ is bounded in $L^\infty(\Omega)$ and in $H^1_0(\Omega)$. We may assume 
$\tu(a_n)\to u_0$ in $H^1(\Omega)$,
$\tu(a_n)\to u_0$ in $L^2(\Omega)$ and $\tu(a_n)\to u_0$
a.e.\ in $\Omega$. 
One easily sees
the function $u_0$ is a nonnegative solution of (\ref{b}).
If $u_0$ is not stable it
must be less than or equal to $M$. 

Let $w(a_n)$ be a first eigenfunction of the linearized equation at $\tu(a_n)$, 
$$\Delta w(a_n)+a_nw(a_n)-f'(\tu(a_n))w(a_n)=-\mu(a_n)w(a_n),$$
normalized so $\int [w(a_n)]^2=\int\phi^2$. The sequence $(w(a_n))$ is bounded in $H^1_0(\Omega)$. So we may also assume
$w(a_n)\weak w_0$ in $H^1_0(\Omega)$, $w(a_n)\to w_0$ in $L^2(\Omega)$ and
$w(a_n)\to w_0$ a.e.\ in $\Omega$.
Moreover, by the Dominated Convergence Theorem $\int f'(\tu(a_n))w_0\to
\int f'(u_0)w_0$. Since, by the Implicit Function Theorem, $(\mu(a_n))$ must decrease to zero if $u_0$ is not stable,
$$ 
\Delta w_0+a_0w_0-f'(u_0)w_0=0,
$$ 
with $w_0\geq 0$. In fact with $w_0>0$.
But $u_0\leq M$ so 
$a_0=\lambda_1$.
This completes the proof that we can follow the branch
$(a,\tua)$ all the way down to $\lambda_1$.

Let $a_n\searrow\lambda_1$. 
The sequence $(\tu(a_n))$ is uniformly bounded. 
The fact $\tu(a_n)$ satisfy (\ref{b}) and \cite[Lemma~9.17]{GT} imply the norms
$\|\tu(a_n)\|_{\h}$ are uniformly bounded. 
Thus $(\tu(a_n))$ has a strongly convergent subsequence in $L^p(\Omega)$.
Subtracting equations (\ref{b}) for $\tu(a_n)$ and $\tu(a_m)$ and using \cite[Lemma~9.17]{GT},
$(\tu(a_n))$ has a subsequence which is strongly convergent in $\h$.

It is easy to see using (\ref{al}) the solutions of $F(\lambda_1,u)=0$ are $u=t\phi$ where $t\in\,]-\infty,M]$,
as for those solutions $\int f(u)\phi=0$,
and we already
showed the only branches of solutions of (\ref{b}) reaching $a=\lambda_1$ are $\C_\dagger$ and
the line $\lin$.
This implies $\tu(a)=u_\dagger(a)$ and completes the proof of the lemma.
\end{proof}

\section{Degenerate solutions with Morse index\\ equal to zero}\label{sectiond}

We now turn to (\ref{a}). 
In Theorem~\ref{thmd} we may relax assumption {\bf (a)} to
\begin{enumerate}[{\bf (a)}]
\item[{\bf (a)$'$}] $h\in L^p(\Omega)$.
\end{enumerate}
Here our main result is 
\begin{Thm}[$\D_*$, degenerate solutions with Morse index equal to zero]\label{thmd}
Suppose $f$ satisfies {\rm\bf (i)}-{\rm\bf (iv)}
and $h$ satisfies {\rm\bf (a)$'$}, {\rm\bf (b)}.
The set of degenerate solutions $(a,u,c)$ of {\rm (\ref{a})} with Morse index equal to zero
is a connected one dimensional manifold $\D_*$ of class $C^1$ in $\R\times\h\times\R$.
The manifold is the union of the half line $\{(\lambda_1,t\phi,0):t\in\,]-\infty,M]\}$
with a graph $\{(a,u_*(a),\cc_*(a)):a\in\,]\lambda_1,+\infty[\}$.
The function $\cc_*$ is nonnegative.
\end{Thm}

The half line $(\lambda_1,t\phi,0)$ for $t\in\,]-\infty,M]$
consists of degenerate solutions of (\ref{a}) with Morse index equal to zero. Consider the function
$G:\R^2\times\rr\times\R\times S\to L^p(\Omega)\times L^p(\Omega)$, where
\begin{equation}\label{ss}
\textstyle S=\left\{w\in\h:\int w^2=\int\phi^2\right\},
\end{equation}
$G$ defined by
$$
\begin{array}{rcl}
 G(a,t,y,c,w)&=&(\Delta\tpy+a\tpy-f\tpy-ch,\\
&&\ \Delta w+aw-f'\tpy w).
\end{array}
$$
Then $G(\lambda_1,t,0,0,\phi)=0$ for $t\in\,]-\infty,M]$.
Notice the difference in notation, ${\cal S}$ in (\ref{s})
as opposed to $S$ in (\ref{ss}).

\begin{Lem}[Initial portion of $\D_*$]\label{curve} 
There exists $\sigma>0$ and $C^1$ functions $\A_*:J\to\R$, $\y_*:J\to\rr$, $\cc_*:J\to\R$ and $\w_*:J\to S$,
where 
$J=\,]-\infty,M+\sigma[$, such that 
the map $t\mapsto(\A_*(t),t\phi+\y_*(t),\cc_*(t))$, defined in $J$, 
with $\A_*(t)=\lambda_1$, $\y_*(t)= 0$, $\cc_*(t)=0$,
$w_*(t)=\phi$ for $t\in\,]-\infty,M]$,
parametrizes a curve $\D_*$ of degenerate solutions of~{\rm (\ref{a})}
with Morse index equal to zero.
There exists a neighborhood of $\D_*$
in $\R\times\h\times\R$ such that the degenerate
solutions of the equation in this neighborhood lie on $\D_*$.
The degenerate directions are given by $\w_*$.
\end{Lem}
\begin{proof}
Let $t_0\in\,]-\infty,M]$.
  We use the Implicit Function Theorem
to show we may write the solutions of $G(a,t,y,c,w)=0$, in a neighborhood
of $(\lambda_1,t_0,0,0,\phi)$, in the form $(\A_*(t),t,\y_*(t),\cc_*(t),\w_*(t))$. Let 
$(\alpha,z,\gamma,\omega)\in\R\times\rr\times\R\times\rr$.
The derivative $DG_{(a,y,c,w)}(\alpha,z,\gamma,\omega)$ at $(\lambda_1,t_0,0,0,\phi)$ is
$$
\begin{array}{rcl}
 G_a\alpha+G_yz+G_c\gamma+G_w\omega\ &=&(\alpha\tpy+\Delta z+az-f'\tpy z-\gamma h,\\
&&\!\!\!\!\!\!\!\!\!\!\!\!\!\!\!\!\!\ \alpha w-f''\tpy zw+\Delta \omega+a\omega-f'\tpy\omega)\\
&\!\!\!\!\!\!\!\!\!\!\!\!\!\!\!\!\!\!\!\!\!\!\!\!\!\!\!\!\!\!\!\!\!\!\!\!=&
\!\!\!\!\!\!\!\!\!\!\!\!\!\!\!\!\!(\alpha t_0\phi+\Delta z+\lambda_1z-\gamma h,
\alpha\phi+\Delta \omega+\lambda_1\omega).
\end{array}
$$
We check this derivative is injective. 
Consider the system obtained by setting the previous derivative equal to 
$(0,0)$.
Multiplying both sides of the second equation by $\phi$ and integrating we get $\alpha=0$.
Thus $\Delta \omega+\lambda_1\omega=0$. Since $\omega\in\rr$ it follows $\omega= 0$.
Now we use the first equation. Multiplying both sides of $\Delta z+\lambda_1z-\gamma h=0$
by $\phi$ and integrating we get $\gamma=0$. This implies $z= 0$ and
proves injectivity. It is easy to check that the derivative is also surjective.
So the derivative is a homeomorphism from $\R\times\rr\times\R\times\rr$ to $L^p(\Omega)\times L^p(\Omega)$.
\end{proof}
Let $\sigma$ be as in Lemma~\ref{curve}.
Because $y_*$ is $C^1$ and $\y_*'(M)=0$,
reducing $\sigma$ if necessary,
$$u_*(t):=t\phi+\y_*(t)$$ is positive and $\max_\Omega u_*(t)>M$ for $t\in\,]M,M+\sigma[$.
Moreover, for $t$ in this interval, $\A_*(t)>\lambda_1$ and $\cc_*(t)>0$ because
$$
(a-\lambda_1)\int w\phi=\int f'(u)w\phi
$$
and
\begin{equation}\label{signc}
\int(f'(u)u-f(u))w=c\int hw.
\end{equation}
We now continue the branch of degenerate solutions $\D_*$ using $a$ as a parameter.
\begin{proof}[Proof of {\rm Theorem~\ref{thmd}}]
Consider the function $H:\R\times\h\times\R\times S\to L^p(\Omega)\times L^p(\Omega)$,
defined by
$$
H(a,u,c,w)=(\Delta u+au-f(u)-ch,\Delta w+aw-f'(u)w).
$$
We know $H(\A_*(t),u_*(t),\cc_*(t),\w_*(t))=0$, 
in particular for $t\in\,]M,M+\sigma[$.
 We use the Implicit Function Theorem
to show that we may write the solutions of $H(a,u,c,w)=0$, in a neighborhood
of a solution $(a_0,u_0,c_0,w_0)$, with $a_0>\lambda_1$, in the form $(a,u_*(a),\cc_*(a),\w_*(a))$. 
[As above, $u_*(t)$ should be called $\tilde{u}_*(t)$ where $\tilde{u}_*(t)=u_*(a_*(t))$, and
similarly for $c_*(t)$ and $w_*(t)$].
Let 
$(v,\gamma,\omega)\in\h\times\R\times\rr_{w_0}$, with 
$$\rr_w:={\textstyle\left\{\omega\in\h:\int \omega w=0\right\}}.$$
The derivative $DH_{(u,c,w)}(v,\gamma,\omega)$ at $(a_0,u_0,c_0,w_0)$ is
$$
\begin{array}{rcl}
 H_uv+H_c\gamma+H_w\omega\ &=&(\Delta v+a_0v-f'(u_0)v-\gamma h,\\
&&\ -f''(u_0) vw_0+\Delta \omega+a_0\omega-f'(u_0)\omega).
\end{array}
$$
We check that the derivative is injective. 
Consider the system obtained by setting the previous derivative equal to 
$(0,0)$.
Multiplying both sides of the first equation by $w_0$ and integrating by parts we obtain $\gamma=0$.
Thus $\Delta v+a_0v-f'(u_0)v=0$. 
So $v=\kappa w_0$ where $\kappa\in\R$. Substituting into the second equation,
$$
-\kappa f''(u_0)w_0^2+\Delta \omega+a_0\omega-f'(u_0)\omega=0.
$$
Multiplying both sides of this equation by $w_0$ and integrating by parts it follows
$$
\kappa\int f''(u_0)w_0^3=0.
$$
This implies that either $\kappa=0$ or $f''(u_0)\equiv 0$.
The function $f''(u_0)$ is identically zero iff $u_0\leq M$. In this case $\Delta w_0+a_0w_0=0$.
Then $a_0=\lambda_1$ because $w_0>0$. 
So $\kappa=0$ and $\Delta \omega+a_0\omega-f'(u_0)\omega=0$. The function $\omega$ has to be a multiple
of $w_0$. On the other hand, $\omega\in\rr_{w_0}$. Finally, $\omega= 0$.
We have proved injectivity. It is easy to check that the derivative is also surjective.
So the derivative is a homeomorphism from $\h\times\R\times\rr_{w_0}$ to $L^p(\Omega)\times L^p(\Omega)$.

The branch of degenerate solutions $(\A_*(t),t\phi+\y_*(t),\cc_*(t))$ for $t\in\,]M,M+\sigma[$ may also
be represented as $(a,u_*(a),\cc_*(a))$. The inverse of $\left.a_*\right|_{]M,M+\sigma[}$ is $t_*$, with
\begin{equation}\label{ta}
\tau(a):=\textstyle\frac{\int u_*(a)\phi}{\int\phi^2}.
\end{equation}
It follows that $t_*(a)$ is an increasing function
of $a$ for $a\in\,\bigl]\lambda_1,\lim_{t\nearrow M+\sigma}\A_*(t)\bigr[$. 

Equality (\ref{signc}) shows that for any degenerate solution with Morse index equal to zero
the value of $c$ is nonnegative.
As $\cc_*(a)\geq 0$, we have $\max_\Omega u_*(a)\leq K_a$ ($K_a$ as in (\ref{ka})).
Moreover, we just saw that $\max_\Omega u_*(a)$ cannot drop to $M$ for $a>\lambda_1$.
We may follow the branch of degenerate solutions $(a,u_*(a),\cc_*(a),\w_*(a))$ while 
$\cc_*(a)$ does not go to $+\infty$ and $\min_\Omega u_*(a)$ does not go to $-\infty$.
Suppose the branch of degenerate solutions exists for $a\in\,]\lambda_1,L[$. Let $a_n\nearrow L$.
We wish to prove that, modulo a subsequence, $\lim_{n\to+\infty}u_*(a_n)$ and $\lim_{n\to+\infty}\cc_*(a_n)$
exist. This will imply that the branch can be extended for all $a\in\,]\lambda_1,+\infty[$.

By a computation similar to the one leading to (\ref{bdd}),
$[u_*(a_n)]^+$ is bounded in $H^1_0(\Omega)$.
Modulo a subsequence, $[u_*(a_n)]^+\weak u_L$ in $H^1_0(\Omega)$, $[u_*(a_n)]^+\to u_L$ in $L^2(\Omega)$ and 
$[u_*(a_n)]^+\to u_L$ a.e. in $\Omega$.
Arguing as in the proof of Lemma~\ref{unique}, 
$\w_*(a_n)\weak w_L$ in $H^1_0(\Omega)$ where
$w_L$ satisfies $\Delta w_L+Lw_L-f'(u_L)w_L=0$, and so $w_L>0$.
Differentiating both sides of $H(a, u_*(a), \cc_*(a), \w_*(a))=0$ with respect to $a$, we readily deduce
$$
\cc_*'(a)=\frac{\int u_*(a)w_*(a)}{\int hw_*(a)}.
$$
Since the sequence $(u_*(a_n))$ is uniformly bounded above
and $w_L$ is positive, $\limsup_{n\to\infty}\cc_*'(a_n)$ is finite. Recalling that $c_*$ is nonnegative,
we conclude that $\lim_{n\to\infty}\cc_*(a_n)$ exists. We call the limit $c_L$.

Next we show $\|u_*(a_n)\|_{L^2(\Omega)}$ are uniformly bounded. By contradiction suppose
$\|u_*(a_n)\|_{L^2(\Omega)}\to+\infty$. Define $v_*(a_n)=u_*(a_n)/\|u_*(a_n)\|_{L^2(\Omega)}$.
The new functions satisfy
$$
\Delta v_*(a_n)+a_nv_*(a_n)-\frac{f(u_*(a_n))}{\|u_*(a_n)\|_{L^2(\Omega)}}-\frac{\cc_*(a_n)}{\|u_*(a_n)\|_{L^2(\Omega)}}h=0.
$$
Multiplying both sides by $v_*(a_n)$ and integrating we conclude $(v_*(a_n))$ is bounded in $H^1_0(\Omega)$.
Modulo a subsequence, $v_*(a_n)\weak v$ in $H^1_0(\Omega)$, $v_*(a_n)\to v$ in $L^2(\Omega)$ and 
$v_*(a_n)\to v$ a.e. in $\Omega$, where
$$
\Delta v+Lv=0.
$$
The function $v$ is nonpositive because $(u_*(a_n))$ is uniformly bounded above.
The function $v$ is negative because it has $L^2(\Omega)$ norm equal to one.
So $L=\lambda_1$, which is a contradiction. For use below, we observe that even in the case $L=\lambda_1$ we
are lead to a contradiction, as we now see. 
Indeed, \cite[Lemma~9.17]{GT} implies $\|v_*(a_n)\|_{\h}$ are uniformly bounded.
We may assume $v_*(a_n)\to v$ in $C^{1,\alpha}(\bar{\Omega})$. Hopf's lemma implies $v_*(a_n)$, and hence $u_*(a_n)$,
are negative for large $n$. For these $n$, the linearized equations become
$$
\Delta\w_*(a_n)+a_n\w_*(a_n)=0.
$$
This is a contradiction as $a_n>\lambda_1$ and the $\w_*(a_n)$ are positive.
We have proved $\|u_*(a_n)\|_{L^2(\Omega)}$ are uniformly bounded.

We use the fact that $(a,u_*(a),\cc_*(a))$ satisfies (\ref{a}) and, 
once more, \cite[Lemma~9.17]{GT} to
guarantee that, modulo a subsequence, $(u_*(a_n))$ converges strongly in $\h$ to a function we designate by $u_*(L)$,
as it satisfies (\ref{a}) with $a=L$ and $c=c_L$.
This finishes the proof that the branch of degenerate solutions can be extended to all $a>\lambda_1$,
because we may apply the Implicit Function Theorem to $H=0$ at the solution $(L,u_*(L),c_L,w_L)$.

To finish the proof of Theorem~\ref{thmd}, we notice that if $(a,\tu(a),\tilde{c}(a))$ is a degenerate solution
of (\ref{a}) with Morse index equal to zero, then, again by the Implicit Function Theorem,
we can follow this solution backwards using the parameter $a$ until we reach $\lambda_1$. Let $a_n\searrow\lambda_1$.
The norms $\|\tu(a_n)\|_{L^2(\Omega)}$ are uniformly bounded, as we saw two paragraphs above.
Arguing as before, $\tilde{c}(a_n)\to \tilde{c}_0$.
The sequence $(a_n,\tu(a_n),\tilde{c}(a_n))$ will converge in $\R\times\h\times\R$ to, say, $(\lambda_1,\tu_0,\tilde{c}_0)$,
solution of (\ref{a})
(see the end of the proof of Lemma~\ref{unique}).
Multiplying both sides of (\ref{a}) by $\phi$ and integrating $\tilde{c}_0=0$.
By Lemma~\ref{curve}
we have $\tu_0=M\phi$ and so $\tu(a)=u_*(a)$.
\end{proof}

\begin{Rmk}\label{unico}
Assume {\bf (a)}.
  If $(a_n,u_n,c_n)$
is a sequence of solutions of {\rm (\ref{a})}, with $a_n>\lambda_1$ and bounded away
from $\lambda_1$, $(a_n)$ bounded above, and $(c_n)$ bounded, then $(u_n)$ is
uniformly bounded.
If $c_n\geq 0$ the same is true under the weaker assumption {\bf (a)$'$}.
\end{Rmk}
Indeed, the hypotheses in the remark imply $(u_n)$ is uniformly bounded above. Therefore
we may argue as in the proof of Theorem~\ref{thmd}.

\begin{Prop}[Behavior of $c$ along $\D_*$]\label{p1} Suppose $f$ satisfies {\rm\bf (i)}-{\rm\bf (iv)}
and $h\in\{\hbar\in C^1(\bar{\Omega}):\hbar=0\ {\rm on}\ \partial\Omega\}$ satisfies {\rm\bf (b)}.
 Along $\D_*$ we have
$$
\lim_{a\to+\infty}c_*(a)=+\infty.
$$
\end{Prop}
\begin{proof}
 Suppose $\underline{c}>0$. Fix any $l>\lambda_1$. 
Choose $t>0$ such that
$t\phi\leq u_\dagger(l)$ where $(l,u_\dagger(l))$ is the stable solution of equation (\ref{b}) in Section~\ref{noh}.
Such a $t$ exists by (\ref{texists}).
Next choose $A\geq l$ large enough satisfying
$$
(A-\lambda_1)t\phi-f(t\phi)-\underline{c}h\geq 0.
$$
Take $a>A$. Then $(a,t\phi,\underline{c})$ is a subsolution of (\ref{a}).
And $(a,u_\dagger(a),\underline{c})$ is a supersolution of (\ref{a}). 
The subsolution and the supersolution are ordered,
$t\phi\leq u_\dagger(l)<u_\dagger(a)$, as the solutions $u_\dagger$ are strictly increasing along $\C_\dagger$.
Moreover, neither $(a,t\phi,\underline{c})$ nor $(a,u_\dagger(a),\underline{c})$ is a solution of (\ref{a}).
Let $K_a$ be as in (\ref{ka}). Clearly, $u_\dagger(a)\leq K_a$.
Define
\begin{equation}\label{fk}
f_{K_a}(u)=\left\{\begin{array}{ll}
                      f(u)&{\rm if}\ u\leq K_a,\\
		      f(K_a)+f'(K_a)(u-K_a)&{\rm if}\ u>K_a.
                      \end{array}
\right.
\end{equation}
By \cite[Theorem~2]{BN}, there is a solution $u_0\in H^1_0(\Omega)$ of
\begin{equation}\label{d}
\Delta u+au-f_{K_a}(u)-\underline{c}h=0,
\end{equation}
$t\phi<u_0<u_\dagger(a)$, such that $u_0$ is a local minimum of $I_{K_a}:H^1_0(\Omega)\to\R$,
defined by
\begin{equation}\label{I}
I_{K_a}(u):={\textstyle\frac{1}{2}}\int(|\nabla u|^2-au^2)+\int F_{K_a}(u)+\underline{c}\int hu.
\end{equation}
Here $F_{K_a}(u)=\int_0^uf_{K_a}(s)\,ds$. By \cite[Lemma~9.17]{GT} $u_0\in\h$. Since $\h\subset H^1_0(\Omega)$,
differentiating $I_{K_a}$ twice, $u_0$ is either stable 
or degenerate 
with Morse index equal to zero in $\h$. As $u_0<u_\dagger(a)<K_a$, and $f_{K_a}$ (respectively $f_{K_a}'$)
coincides with $f$ (respectively $f'$) below $K_a$, $u_0$ is a solution of (\ref{a}), 
either stable or degenerate 
with Morse index equal to zero.
By the Theorem~\ref{big} ahead, $(a,u_0,\underline{c})$ must lie on ${\cal M}^*\cup\{{\bm p}_*\}$ and
$c_*(a)\geq\underline{c}$. We have shown that
given $\underline{c}>0$, there exists $A>0$, such that for all $a>A$ we have
$c_*(a)\geq\underline{c}$. The proof is complete.
\end{proof}

\section{Stable solutions, solutions around $\D_*$,
solutions around zero, and mountain pass solutions}\label{four}

In this section we treat successively
stable solutions in Theorem~\ref{big} and Proposition~\ref{p2}, solutions around $\D_*$ in Lemma~\ref{turn},
solutions around zero in Lemma~\ref{p3}, and mountain pass solutions in Lemma~\ref{mp}. We finish 
by proving the existence of at least three solutions for $a>\lambda_2$, $a$ not an eigenvalue, and small $|c|$ in Theorem~\ref{i1}.

\begin{Thm}[${\cal M}^*$, stable solutions of {\rm (\ref{a})}]\label{big} 
Suppose $f$ satisfies {\rm\bf (i)}-{\rm\bf (iv)}
and $h$ satisfies {\rm\bf (a)-(b)}. Fix $a>\lambda_1$.
The set of stable solutions $(c,u)$ of {\rm (\ref{a})} is 
a $C^1$ manifold ${\cal M}^*$, which is the graph $\{(c,u^*(c)):c\in\,]-\infty,c_*[\}$. 
The limit $\lim_{c\nearrow c_*}(c,u^*(c))$ exists
and  equals ${\bm p}_*:=(c_*,u_*)$, the degenerate solution with Morse index equal to zero
on the curve $\D_*$ (i.e.\ $(c_*,u_*)=(c_*(a),u_*(a))$
of {\rm Theorem~\ref{thmd}}). 
If $c_1<c_2$, then $u^*(c_1)>u^*(c_2)$.
\end{Thm}
\begin{proof}[Sketch of the proof]\footnote{For the full proof see the Appendix~\ref{appendix}.}
The argument is very close to \cite[proof of Theorem~3.2]{OSS1}.
 Consider the function $\tH:\h\times\R\times S\times\R\to L^p(\Omega)\times L^p(\Omega)$ ($S$ given in (\ref{ss})),
$\tH$ defined by
$$
\tH(u,c,w,\mu)=(\Delta u+au-f(u)-ch,\Delta w+aw-f'(u)w+\mu w).
$$
Apply the Implicit Function Theorem to describe the solutions of $\tH=0$
in a neighborhood of a stable solution $(u,c,w,\mu)$. 
Here $\mu$ is the first eigenvalue of the associated linearized problem and $w$ is
the corresponding positive eigenfunction on $S$.
For each fixed $a$, the solutions of $\tH=0$ in a neighborhood of a stable solution $(u,c,w,\mu)$ lie
on a $C^1$ curve parametrized by $c\mapsto(\uu^*(c),c,\w^*(c),\mu^*(c))$.
Differentiate both sides of the equations $\tH(\uu^*(c),c,\w^*(c),\mu^*(c))=(0,0)$
with respect to $c$. When $\mu>0$, using the maximum principle,
$$ 
(\uu^*)'<0
$$ 
and
$$ 
(\mu^*)'=\frac{\int f''(u^*)(u^*)'(w^*)^2}{\int\phi^2}<0.
$$ 
By Remark~\ref{unico}, we may follow the solution $\uu^*(c)$ until it becomes degenerate. 
The solution $\uu^*(c)$ will have to become degenerate for some value of $c$.
Indeed, from (\ref{a}) we obtain
\begin{equation}\label{upc}
 c\int h\phi=(a-\lambda_1)\int u\phi-\int f(u)\phi\leq (a-\lambda_1)\int u^+\phi,
\end{equation}
showing $c$ is bounded above. Thus, the solutions $\uu^*(c)$ cannot be continued for all positive values of $c$.
There must exist $c_*$ such that $\lim_{c\nearrow c_*}\mu^*(c)=0$. Clearly, the solutions $\uu^*(c)$ will converge
to a solution $u_*$ as $c\nearrow c_*$. By the uniqueness assertion in Theorem~\ref{thmd},
this $(c_*,u_*)$ belongs to $\D_*$. In particular $c_*>0$. 

The above shows any branch of stable solutions can be extended for $c\in\,]-\infty,c_*[$.
But by Lemma~\ref{unique} there is a unique stable solution of (\ref{a}) for $c=0$,
namely $u_\dagger=u_\dagger(a)$. This proves uniqueness.
\end{proof}

\begin{Prop}[Stable solutions are superharmonic for small $|c|$]\label{p2}
Suppose $f$ satisfies {\rm\bf (i)}-{\rm\bf (iv)}
and $h\in\{\hbar\in C^1(\bar{\Omega}):\hbar=0\ {\rm on}\ \partial\Omega\}$
satisfies {\rm\bf (b)}.
Let $a>\lambda_1$ and $u^*$ be as in {\rm Theorem~\ref{big}}.
For small $|c|$,
\begin{equation}\label{ch}
au^*(c)-f(u^*(c))>ch.
\end{equation}
\end{Prop}
This generalizes formula (3.2) of \cite{OSS1}. The proof of \cite{OSS1} does not carry through
in our setting however.
\begin{proof}[Proof of {\rm Proposition~\ref{p2}}]
Fix $a>\lambda_1$ and let $u_\dagger=u_\dagger(a)$ be the stable solution of (\ref{a}) for $c=0$.
In the proof of Theorem~\ref{thmcz} we showed $0<u_\dagger<K_a$.
So $au_\dagger-f(u_\dagger)$ is strictly positive on $\Omega$. Hopf's lemma implies
$\frac{\partial u_\dagger}{\partial n}<0$ on $\partial\Omega$, where $n$ denotes
the exterior unit normal to $\partial\Omega$. As $f'(0)=0$,
we also have $\frac{\partial}{\partial n}(au_\dagger-f(u_\dagger))<0$ on $\partial\Omega$.
Since the function $\uu^*$ in Theorem~\ref{big} is $C^1$ into $\h$, and so in particular 
continuous, and $\h$ is compactly embedded in $C^{1,\alpha}(\bar{\Omega})$,
it follows $\frac{\partial}{\partial n}(a\uu^*(c)-f(\uu^*(c)))<0$ on $\partial\Omega$
for small $|c|$. Moreover, $a\uu^*(c)-f(\uu^*(c))>0$ in $\Omega$ for small $|c|$. 
If $h$ belongs to the space $C^1(\bar{\Omega})$ and vanishes
on $\partial\Omega$, then we obtain (\ref{ch}), for small $|c|$.
\end{proof}

\begin{Lem}[Solutions around $\D_*$, {\cite[Theorem~3.2]{CR2}}, {\cite[p.~3613]{OSS1}}]\label{turn} Let $a>\lambda_1$ be fixed and
${\bm p}_*=(c_*,u_*)$ be a degenerate solution with Morse index equal to zero.
There exists a neighborhood of ${\bm p}_*$ in $\R\times\h$ such that the
set of solutions of {\rm (\ref{a})} in the neighborhood is a $C^1$ manifold.
This manifold is ${\bm m}^\sharp\cup\{{\bm p}_*\}\cup{\bm m}^*$. Here
\begin{itemize}
\item
${\bm m}^\sharp$ is a manifold of nondegenerate solutions with Morse index equal to one, which
is a graph $\{(c,u^\sharp(c)):c\in\,]c_*-\eps_*,c_*[\}$.
\item ${\bm m}^*$ is a manifold of stable solutions, which
is a graph $\{(c,u^*(c)):c\in\,]c_*-\eps_*,c_*[\}$.
\end{itemize}
The value $\eps_*$ is positive. The manifolds ${\bm m}^\sharp$ and ${\bm m}^*$ are connected by $\{{\bm p}_*\}$.
\end{Lem}
\begin{proof}[Sketch of the proof]\footnote{For the full proof see the Appendix~\ref{appendix}.}
 This lemma is known, but we sketch a proof
 adapted to our framework.
Let $(c_*,u_*)$ be a degenerate solution with Morse index equal to zero.
Let 
$t_*$ and $y_*$ be such that
$u_*=t_*w_*+y_*$, with $w_*\in S$ satisfying $\Delta w_*+aw_*-f'(u_*)w_*=0$, $w_*>0$, and $y_*\in\rr_{w_*}$.
We let $\tG:\R\times\rr_{w_*}\times\R\times S\times\R\to L^p(\Omega)\times L^p(\Omega)$ be defined by
\begin{eqnarray*}
\tG(t,y,c,w,\mu)&=&(\Delta\twzy+a\twzy-f\twzy-ch,\\ &&\ \Delta w+aw-f'\twzy w+\mu w).
\end{eqnarray*}
We may use the Implicit Function Theorem to describe the solutions of $\tG=0$
in a neighborhood of $(t_*,y_*,c_*,w_*,0)$. They lie
on a curve $t\mapsto(t,\y(t),\cc(t),\w(t),\mu(t))$. 
It is impossible for $\max_\Omega\twyz\leq M$ because otherwise $\Delta w_*+aw_*=0$ with $w_*>0$.
Differentiating $$\tG(t,\y(t),\cc(t),\w(t),\mu(t))=(0,0)$$
with respect to $t$,
$$
\mu'(t_*)=\frac{\int f''\twyz w_*^3}{\int w_*^2}>0,
$$
$\cc'(t_*)=0$ and
$$
\cc''(t_*)=-\,\frac{\int f''\twyz w_*^3}{\int hw_*}<0.
$$
This is formula (2.7) of \cite{OSS1}.
We recall from the proof of Theorem~\ref{thmd}, a degenerate solution
with $a>\lambda_1$ has $c_*>0$. As $t$ increases from $t_*$,
$\cc(t)$ decreases and the solution becomes stable. 
So the ``end" of ${\cal M}^*$ coincides with the piece of curve parametrized by
$t\mapsto(c(t),tw_*+\y(t))$, for $t$ in a right neighborhood of $t_*$.
A parametrization of ${\bm m}^\sharp$ is obtained by taking $t$ in a left neighborhood of $t_*$.
\end{proof}

Let $-\infty=:\lambda_0<\lambda_1<\lambda_2<\ldots<\lambda_i<\ldots$,  
denote the eigenvalues of the Dirichlet Laplacian on $\Omega$.

\begin{Lem}[Solutions around zero, {\cite[Theorem~3.3]{OSS1}}, {\cite[Theorem~2]{CP}}]\label{p3}
 Let $\lambda_i<a<\lambda_{i+1}$ for some $i\geq 0$. 
There exists
a $C^1$ function $\breve{u}$ defined in $]-\breve{c},\breve{c}[$
such that $c\mapsto(c,\breve{u}(c))$ parametrizes a curve $\breve{\C}$ of solutions of {\rm (\ref{a})} passing at
$(0,0)$. The solutions on $\breve{\C}$ have Morse index equal to 
the sum of the dimensions of the eigenspaces corresponding to $\lambda_1$ to $\lambda_i$. 
If $a$ is sufficiently close to $\lambda_1$,
$c_1<c_2$ implies $\breve{u}(c_1)<\breve{u}(c_2)$.
\end{Lem}
We refer to the cited works for the proof of this lemma.\footnote{For the full proof see Appendix~\ref{appendix}.}
 The assertion on the Morse index follows from its continuity on $\breve{u}$.

\begin{Lem}[Mountain pass solutions, \cite{CT}]\label{mp}
 Let $a>\lambda_1$ and $c<c_*(a)$. Then {\rm (\ref{a})} has at least two solutions, the stable solution and 
 a mountain pass solution.
\end{Lem}
\begin{proof} 
Choose $K=K(a,c)$ such that
\begin{equation}\label{kkk}
f'(K)>a\quad {\rm and}\quad (u\geq K\ \Rightarrow\ au-f(u)-ch\leq 0\ {\rm in}\ \Omega).
\end{equation}
Define $f_K$ by (\ref{fk}) with $K_a$ replaced by $K$.
By the maximum principle, $(a,u,c)$ is a solution of (\ref{a}) iff
$(a,u,c)$ is a solution of (\ref{d}) (with $f_{K_a}$ replaced by
$f_K$ and $\underline{c}$ replaced by $c$). It is easy to see that
$I_K$, defined by (\ref{I}), satisfies the Palais-Smale condition.
Indeed, $I_K'(u_n)\to 0$ implies $(\|u_n\|_{H^1_0(\Omega)})$ is bounded,
and then $(I_K(u_n))$ convergent implies $(u_n)$ has a convergent subsequence in $H^1_0(\Omega)$.
Also, $I_K(t\phi)\to-\infty$ as $t\to-\infty$.
As for the solution $(a,u^*(c),c)$ of (\ref{a}) we have $u^*(c)\leq K$, and $f$ (respectively $f'$)
coincides with $f_{K}$ (respectively $f_{K}'$) below $K$, $(a,u^*(c),c)$ is a stable in $\h$ solution of (\ref{d}).
It is also stable in $H^1_0(\Omega)$ because of \cite[Lemma~9.17]{GT}.
By the Mountain-Pass Lemma (\cite[Theorem~2.1]{AR}), there exists a solution $(a,u_1,c)$ of (\ref{d}).
By the maximum principle $u_1\leq K$. Again because
$f_{K}$ 
coincides with $f$ below $K$, 
$(a,u_1,c)$ is a solution of (\ref{a}).
\end{proof}
\begin{Thm}[Existence of at least three solutions for $\lambda_2<a\neq\lambda_i$ and small $|c|$]\label{i1}
 Suppose $f$ satisfies {\rm\bf (i)}-{\rm\bf (iv)}
and $h$ satisfies {\rm\bf (a)}.
 If $\lambda_i<a<\lambda_{i+1}$, for some $i\geq 2$, and $|c|$ is small, then {\rm (\ref{a})} has at least three solutions.
\end{Thm}
Note here we do not need to assume $h$ satisfies {\rm\bf (b)}.
\begin{proof}[Proof of {\rm Theorem~\ref{i1}}]
Let $c<c_*(a)$.
From \cite[Theorem~10.15]{R},
the Morse index of $(a,u_1,c)$ is at most equal to one in $H^1_0(\Omega)$, and 
$(a,u_1,c)$ is degenerate if it has Morse index equal to zero. 
In the second case the Morse index of $(a,u_1,c)$ in $\h$ is zero and this solution is either
stable or degenerate in $\h$. But it cannot be stable in $\h$ because our analysis implies that, for fixed
$a$, there is at most one stable solution for each $c$ and $u_1\neq u^*(c)$.
It cannot be degenerate either because $c<c_*(a)$.
We conclude the Morse index of $(a,u_1,c)$ in $\h$ has to be equal to one.

Let $|c|<\min\{\breve{c},c_*\}$. Then (\ref{a}) has the stable solution (with Morse index equal to zero),
the mountain pass solution (with Morse index equal to one) and 
the solution $\breve{u}$ of Lemma~\ref{p3} (with Morse index at least equal to 2). 
\end{proof}
In case $\lambda_1<a<\lambda_2$, it follows from Theorem~\ref{l1l2}, which we will prove ahead, we have
$u_1=u^\sharp$. Furthermore, for small $|c|$, we have $u_1=\breve{u}$.

\section{Global bifurcation below $\lambda_2$}\label{below_l2}

In this section we obtain global bifurcation curves below $\lambda_2$.
We briefly treat the case $a\leq\lambda_1$.
Then we examine the situation in a right neighborhood of $\lambda_1$.
Finally we prove Theorem~\ref{l1l2}.

Let $(a,u,c)$ be a solution of (\ref{a}).
Consider the quadratic form
\begin{equation}\label{q}
Q_a(v):=\int\left(|\nabla v|^2-av^2+f'(u)v^2\right).
\end{equation}
For any $v\in\h$ with $L^2(\Omega)$ with norm equal to one,
$Q_a(v)\geq\lambda_1-a$. If $a<\lambda_1$, then $u$ is stable.
The study of the bifurcation curves for $a\leq\lambda_1$ is simple.
We only draw the final pictures in Figures~\ref{fig4} and \ref{fig5}.
If one decreases $a$ below $\lambda_1-\delta$ the linear part of the
curve for $c>0$ starts bending. This evolution is similar to the one
that happens from Figure~\ref{fig6} to Figure~\ref{fig1}. 
We should mention
the linearity of the branches for $c>0$ in Figure~\ref{fig4} is a
consequence of 
$$
u=c\left[\textstyle\frac{\int h\phi}{a-\lambda_1}\phi+(\Delta+a)^{-1}\left[h-\left(\int h\phi\right)\phi\right]\right],
$$
valid when the right hand side is less than or equal to $M$.
\begin{figure}
\centering
\begin{psfrags}
\psfrag{c}{{\tiny $c$}}
\psfrag{u}{{\tiny $u$}}
\psfrag{M}{{\tiny $M\phi$}}
\psfrag{N}{{\tiny $\!\!\!\!\!\!-\frac{M}{\beta}\phi$}}
\psfrag{d}{{\tiny $\cc(\lambda_2)$}}
\psfrag{a}{{\tiny $u(a)$}}
\includegraphics[scale=.34]{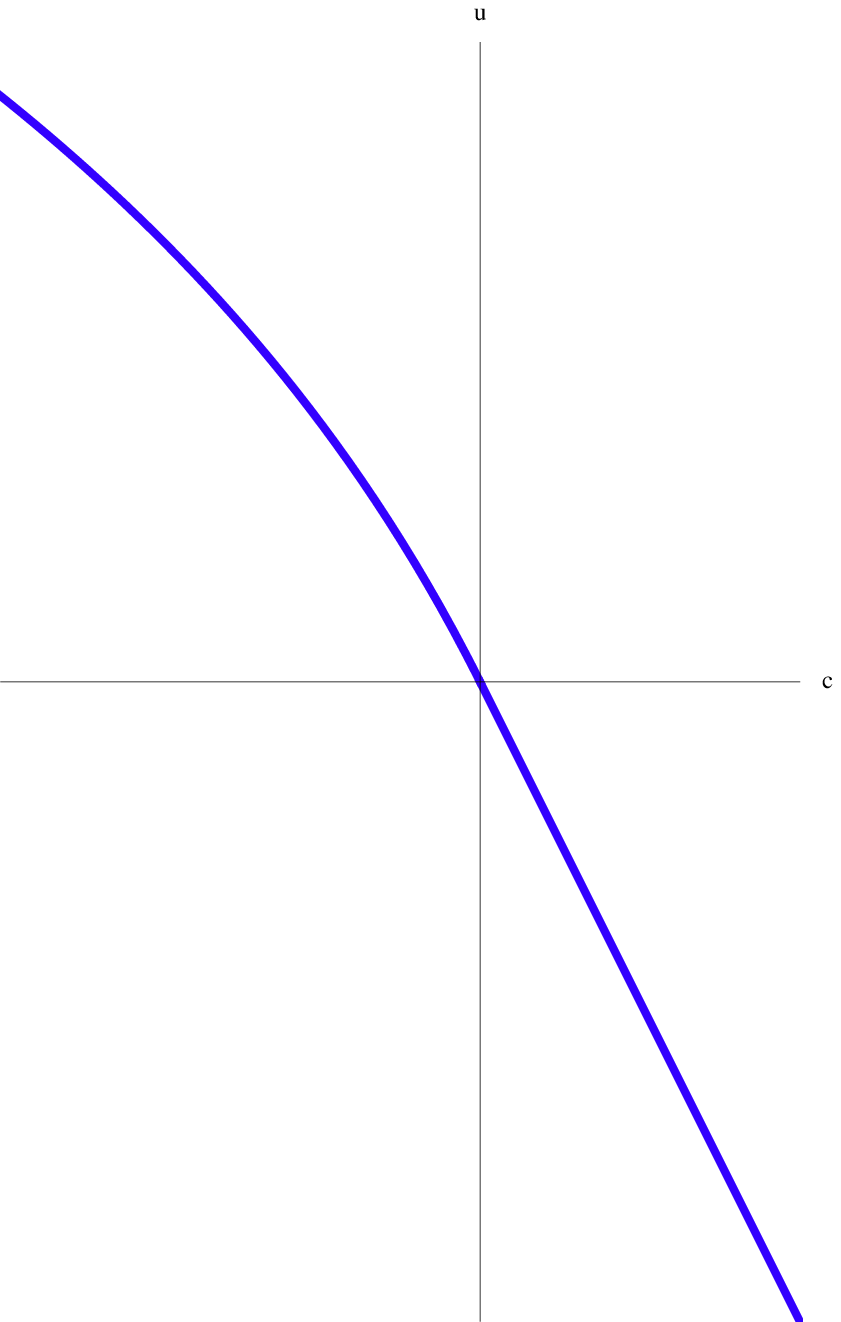}\ \ \ \ \ \ \ \ \ \ \ \ 
\includegraphics[scale=.32]{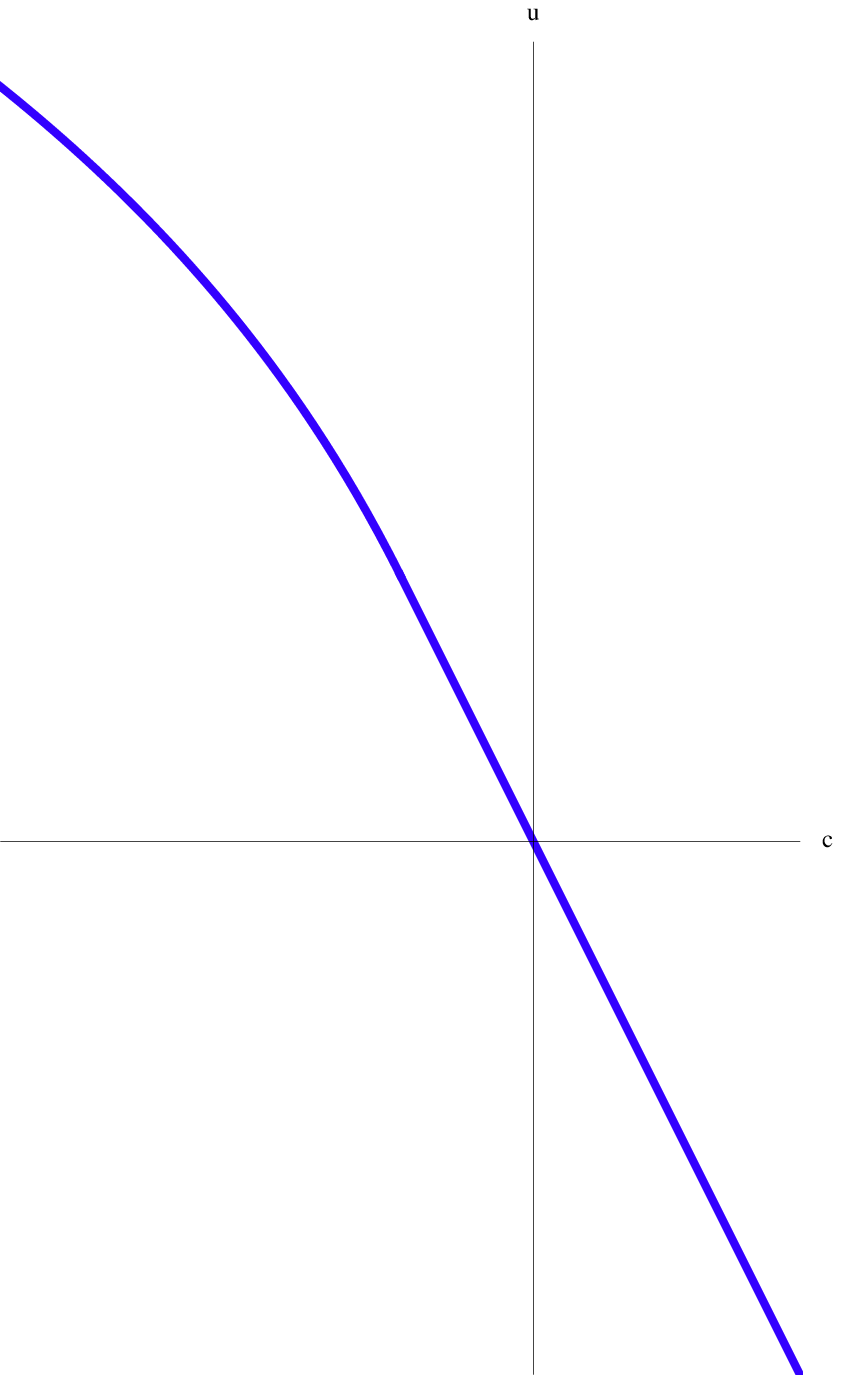}
\end{psfrags}
\caption{Bifurcation curve for $\lambda_1-\delta<a<\lambda_1$. 
On the left $M=0$ and on the right $M>0$.}\label{fig4}
\end{figure}
\begin{figure}
\centering
\begin{psfrags}
\psfrag{c}{{\tiny $c$}}
\psfrag{u}{{\tiny $u$}}
\psfrag{M}{{\tiny $M\phi$}}
\psfrag{N}{{\tiny $\!\!\!\!\!\!-\frac{M}{\beta}\phi$}}
\psfrag{d}{{\tiny $\cc(\lambda_2)$}}
\psfrag{a}{{\tiny $u(a)$}}
\includegraphics[scale=.43]{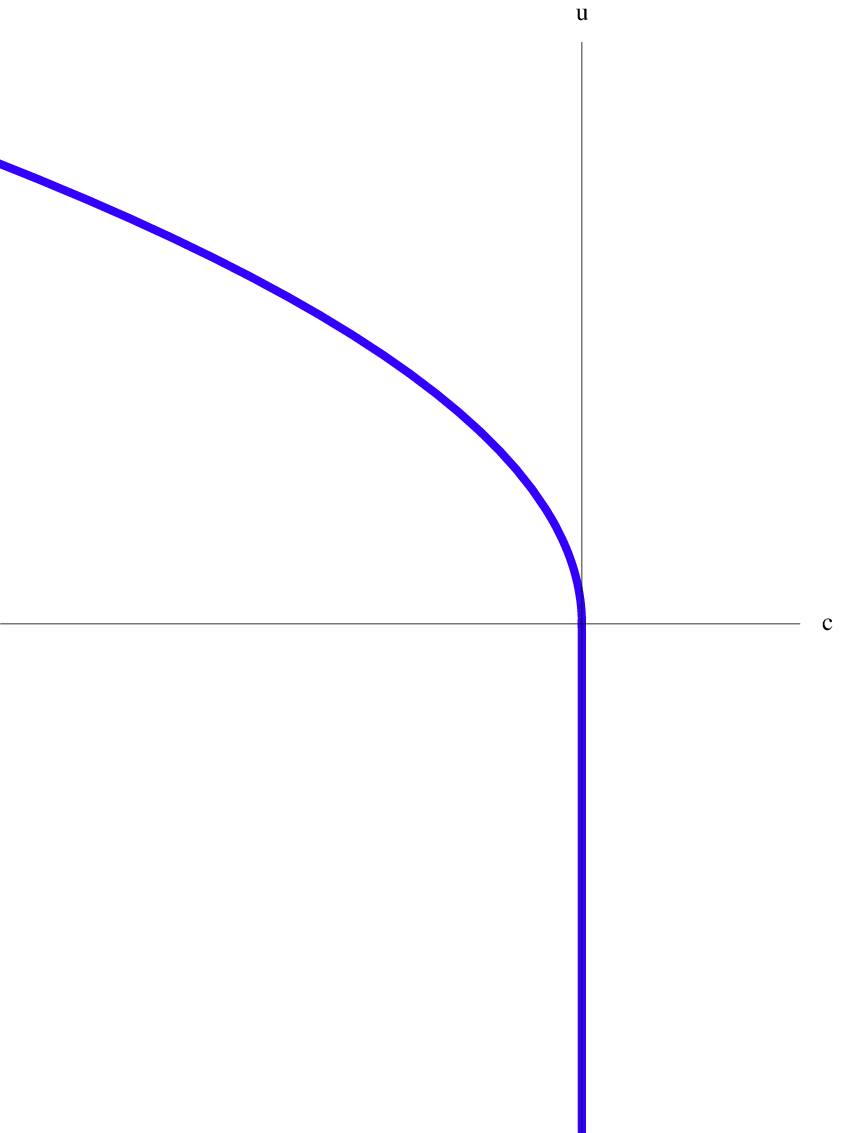}\ \ \ \ \ \ \ \ \ \ \ \ 
\includegraphics[scale=.32]{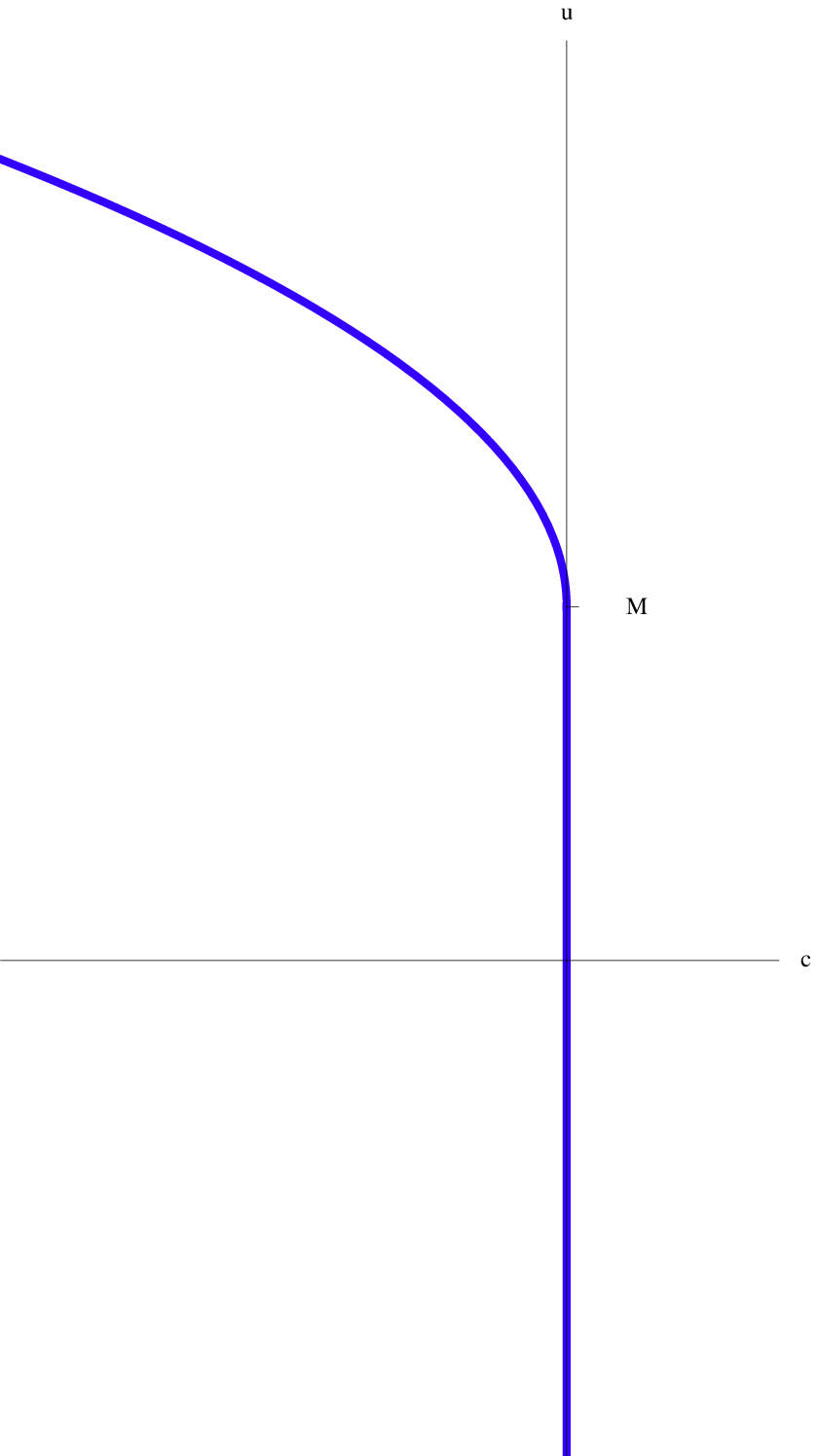}
\end{psfrags}
\caption{Bifurcation curve for $a=\lambda_1$. On the left $M=0$ and on the right $M>0$.}\label{fig5}
\end{figure}

We remark that 
multiplying both sides of (\ref{a}) by $\phi$ and integrating, we see
that when $u$ and $c$ are both nonnegative, $u$ not identically zero, then the constant $a$
is greater than or equal to $\lambda_1$.

We turn to the case $a>\lambda_1$.
\begin{Thm}[$\lambda_1\leq a<\lambda_1+\delta$]\label{global}
Suppose $f$ satisfies {\rm\bf (i)}-{\rm\bf (iv)}
and $h$ satisfies {\rm\bf (a)$'$}-{\rm\bf (b)}.
There exists $\delta>0$ such that the following holds.
The solutions $(a,u,c)$ of {\rm (\ref{a})} for $\lambda_1\leq a<\lambda_1+\delta$ and $c\geq 0$ can be parametrized
in the global chart
$$
\cS=\{(a,t)\in\R^2:\lambda_1\leq a<\lambda_1+\delta\ {\rm and}\ 0\leq t\leq\ttt(a)\}
$$
(with $\ttt(a)$ given in {\rm (\ref{t})}) by
$$
(a,t)\mapsto (a,u_\phi(a,t),\cc_\phi(a,t)),\ \ \  u_\phi(a,t)=t\phi+\y_\phi(a,t).
$$
Here $\y_\phi:\cS\to\rr$ and $\cc_\phi:\cS\to\R$ are $C^1$ functions.
We have 
$u_\phi(\lambda_1,t)=t\phi$ and $\cc_\phi(\lambda_1,t)=0$, for $t\in[0,\ttt(\lambda_1)]=[0,M]$; in addition
$u_\phi(a,0)=0$, $\cc_\phi(a,0)=0$, 
$u_\phi(a,\ttt(a))$ is the stable solution $u_\dagger(a)$ of {\rm (\ref{b})}, and $\cc_\phi(a,\ttt(a))=0$,
for $a\in\,]\lambda_1,\lambda_1+\delta[$.
For each fixed $a$, the map $t\mapsto\cc_\phi(a,t)$ is strictly increasing until the corresponding solution 
lies on the degenerate curve $\D_*$
of {\rm Section~\ref{sectiond}}, and then is strictly decreasing until zero.
The solutions $u_\phi(a,t)$ are strictly increasing with $t$, and so in particular are positive for $t\in\,]0,\ttt(a)]$.
\end{Thm}
\begin{Rmk}\label{rmk3}
In {\rm Theorem~\ref{global}} we may relax assumption {\bf (b)} to
\begin{itemize}
 \item[{\bf (b)$''$}] $\int h\phi>0$.
\end{itemize}
Incidentally, this is also true for {\rm Lemma~\ref{p3}} (where this assumption is used only in connection with
the last assertion).
\end{Rmk}
\begin{Rmk}
When $M=0$, as $a\searrow\lambda_1$ the curve 
$\{(c_\phi(a,t),u_\phi(a,t)):t\in[0,\ttt(a)]\}$ degenerates onto the point $(0,0)$. This case was studied in
{\rm \cite{OSS1}} for $f$ a quadratic function, as mentioned in the {\rm Introduction}. When $M>0$, as $a\searrow\lambda_1$ the curve 
$\{(c_\phi(a,t),u_\phi(a,t)):t\in[0,\ttt(a)]\}$ degenerates onto the segment
$\{(0,t\phi):t\in[0,M]\}$.
\end{Rmk}
Remark~\ref{rmk3} is in line with Theorem~1.3 of \cite{I}
({\bf (b)$''$} corresponds to \cite[formula~(1.9)]{I}).
Theorem~\ref{global} is illustrated in Figure~\ref{fig6}.

\begin{figure}
\centering
\begin{psfrags}
\psfrag{c}{{\tiny $c$}}
\psfrag{u}{{\tiny $u$}}
\psfrag{M}{{\tiny $M\psi$}}
\psfrag{N}{{\tiny $\!\!\!\!\!\!-\frac{M}{\beta}\psi$}}
\psfrag{d}{{\tiny $\cc(\lambda_2)$}}
\psfrag{a}{{\tiny $u(a)$}}
\includegraphics[scale=.3]{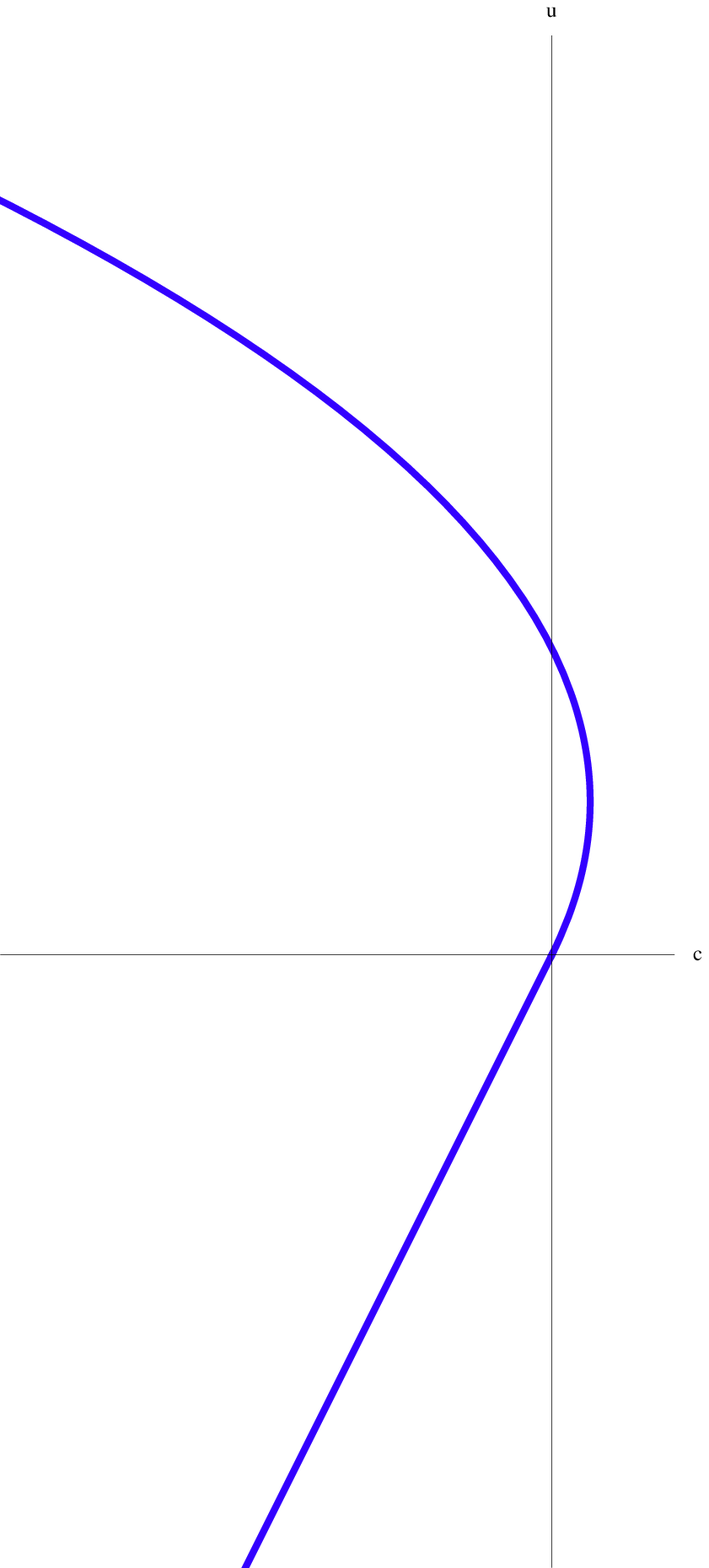}\ \ \ \ \ \ \ \ \ \ \ \ \ \ \ \ \ \ \ \ \ \ \ \ 
\includegraphics[scale=.37]{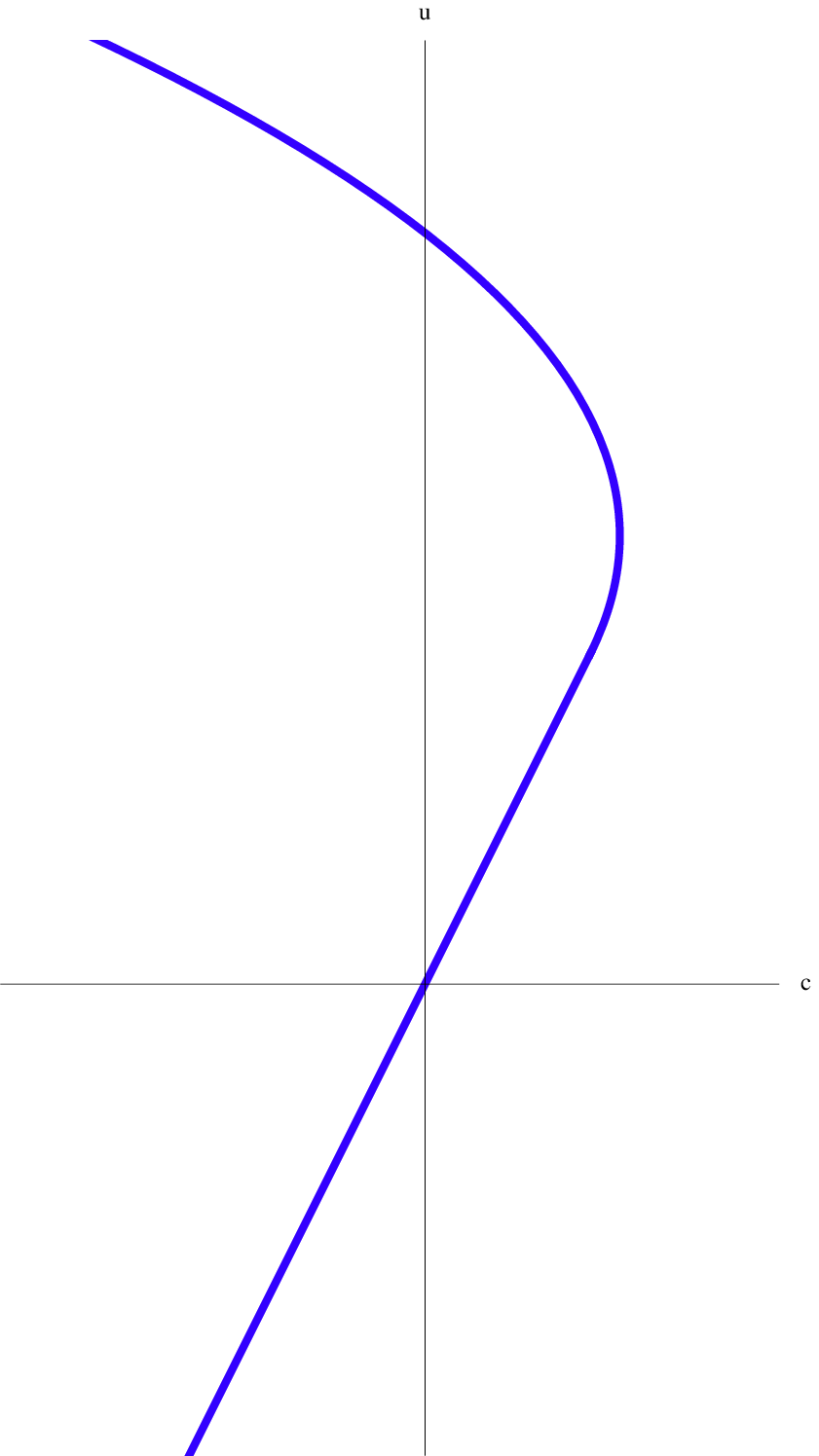}
\end{psfrags}
\caption{Bifurcation curve for $\lambda_1<a<\lambda_1+\delta$. On the left $M=0$ and on the right $M>0$.}\label{fig6}
\end{figure}

\begin{proof}[Proof of {\rm Theorem~\ref{global}}]
We write $u=t\phi+y$, with $y\in\rr$, and
we construct a surface of solutions of (\ref{a}) parametrized by $a$ and $t$.
Let
$$
\tg(a,t,y,c)=\Delta\tpy+a\tpy-f\tpy-ch.
$$
Let $t_0\leq M$.
Clearly $\tg(\lambda_1,t_0,0,0)=0$.  At $(\lambda_1,t_0,0,0)$,
$$
\tg_yz+\tg_c\gamma=\Delta z+\lambda_1 z-\gamma h=0
$$
implies $\gamma=0$ and $z=0$. 
The Implicit Function Theorem guarantees
in a neighborhood of $(\lambda_1,t_0,0,0)$
the solutions of $\tg=0$ lie on a surface
$(a,t)\mapsto(a,t,\y_\phi(a,t),\cc_\phi(a,t))$.
Let now $\lambda_1<a\neq\lambda_i$ for all $i>1$. At $(a,0,0,0)$,
$$
\tg_yz+\tg_c\gamma=\Delta z+az-\gamma h=0
$$
also implies $\gamma=0$ and $z=0$. The Implicit Function Theorem again guarantees
in a neighborhood of $(a,0,0,0)$
the solutions of $\tg=0$ lie on a surface
$(a,t)\mapsto(a,t,\y_\phi(a,t),\cc_\phi(a,t))$. In particular,
we have a surface of solutions defined in a neighborhood $\cN$ of
$$
\{(a,t)\in\R^2:(a=\lambda_1\ {\rm and}\ 0\leq t\leq M)\ {\rm or}\ (\lambda_1\leq a<\lambda_2\ {\rm and}\ t=0)\},
$$
such that the solutions on this surface are the unique solutions of the equation $\tg=0$ in a neighborhood
${\cal V}$ of $\{(\lambda_1,t,0,0):t\in[0,M]\}\cup\{(a,0,0,0):a\in[\lambda_1,\lambda_2[\}$.
Let $\delta>0$ be small enough so that $\cS$
is contained in $\cN$. 
For each $(a,t)\in\cS$, $(a,t\phi+\y_\phi(a,t),\cc_\phi(a,t))$ is a solution of (\ref{a}).

Let $\lambda_1<a<\lambda_1+\delta$ and $u_\dagger(a)$ be the stable solution in Section~\ref{noh}.
Suppose $(a,\ttt(a),u_\dagger(a)-\ttt(a)\phi,0)\in{\cal V}$.
Since $(a,\ttt(a),\y_\phi(a,\ttt(a)),\cc_\phi(a,\ttt(a)))\in{\cal V}$ and there is only one solution in ${\cal V}$
corresponding to each pair $(a,t)$, we have 
$\y_\phi(a,\ttt(a))=u_\dagger(a)-\ttt(a)\phi$, or
$u_\phi(a,\ttt(a)):=\ttt(a)\phi+\y_\phi(a,\ttt(a))=u_\dagger(a)$.
Now $(a,\ttt(a),u_\dagger(a)-\ttt(a)\phi,0)$ will belong to ${\cal V}$ for $a$ close to $\lambda_1$
and both $u_\dagger(a)$ and $u_\phi(a,\ttt(a))$ are continuous. So $u_\phi(a,\ttt(a))=u_\dagger(a)$ 
and $\cc_\phi(a,\ttt(a))=0$ for $\lambda_1<a<\lambda_1+\delta$.

We have $\frac{\partial u_\phi}{\partial t}(\lambda_1,t)=\phi$, for all $t\in[0,M]$, and the function $u_\phi$ is $C^1$.
By reducing $\delta$ if necessary, we may assume $\frac{\partial u_\phi}{\partial t}$ is a strictly positive function 
at each point of $\cS$.
So let $\lambda_1\leq a<\lambda_1+\delta$ and $0\leq t_1<t_2\leq\ttt(a)$. Then $u_\phi(a,t_1)<u_\phi(a,t_2)$, i.e.\ for fixed $a$,
$u_\phi$ is strictly increasing along the curve of solutions joining zero to the stable solution.

We know $\cc_\phi(\lambda_1,t)=0$ for $t\leq M$, and $\cc_\phi(a,0)=0$.
We also know $\cc_\phi(a,\ttt(a))=0$.
Differentiating $\tg(a,t,\y_\phi(a,t),\cc_\phi(a,t))=0$ with respect to $t$,
at $(a,0)$,
$$
\Delta z+az-\gamma h=-(a-\lambda_1)\phi.
$$
Here $z=\frac{\partial y_\phi}{\partial t}$ and $\gamma=\frac{\partial c_\phi}{\partial t}$ at $(a,0)$.
This implies
\begin{equation}\label{c}
\frac{\partial\cc_\phi}{\partial t}(a,0)=\frac{{\textstyle\int \phi^2}}{{\textstyle\int h\phi}}\,\,(a-\lambda_1).
\end{equation}
If we fix $\lambda_1<a<\lambda_1+\delta$ and start increasing $t$ from zero, (\ref{c}) shows that initially $\cc_\phi$ 
increases.
We denote by $(\mu_\phi(a,t),\w_\phi(a,t))$ the first eigenpair of the linearized problem at $u_\phi(a,t)$.
Another application of the Implicit Function Theorem shows this eigenpair has a $C^1$ dependence on
$(a,t)$. Differentiating (\ref{a}) with respect to $t$ and using the definition of $(\mu_\phi,w_\phi)$,
$$
\left\{\begin{array}{rcl}
\Delta v+av-f'(u_\phi)v&=&\gamma h,\\
\Delta w_\phi+aw_\phi-f'(u_\phi)w_\phi&=&-\mu_\phi w_\phi,
\end{array}\right.
$$
where $v=\frac{\partial u_\phi}{\partial t}$ and $\gamma$ is as before.
This implies
\begin{equation}\label{mut}
-\mu_\phi\int {\textstyle\frac{\partial u_\phi}{\partial t}}w_\phi={\textstyle\frac{\partial c_\phi}{\partial t}}\int hw_\phi.
\end{equation}
By (\ref{mut}), $\mu_\phi$ is negative as long as $\frac{\partial c_\phi}{\partial t}>0$. On the other hand the equality $\cc_\phi(a,\ttt(a))=0$,
implies $\frac{\partial c_\phi}{\partial t}(a,\bar{t}(a))=0$ for some $0<\bar{t}(a)<\ttt(a)$.
If $\frac{\partial\cc_\phi}{\partial t}(a,\bar{t}(a))$ vanishes, $(a,u_\phi(a,\bar{t}(a)),\cc_\phi(a,\bar{t}(a)))$
is a degenerate solution with Morse index equal to zero.
By the uniqueness statement of Theorem~\ref{thmd}, it belongs to $\D_*$ and
$\bar{t}(a)=\tau(a)$, where $\tau(a)$ is as in (\ref{ta}).
Now 
$\mu_\phi(a,t)<0$ for $0\leq t<\tau(a)$, and
$\mu_\phi(a,t)>0$ for $\tau(a)<t\leq\ttt(a)$ by Lemma~\ref{turn}, otherwise we would obtain more than one solution
on $\D_*$ for a fixed value of $a$. 
From (\ref{mut}), 
$\frac{\partial c_\phi}{\partial t}(a,t)>0$ for $0\leq t<\tau(a)$, and
$\frac{\partial c_\phi}{\partial t}(a,t)<0$ for $\tau(a)<t\leq\ttt(a)$.
Therefore, $\cc_\phi(a,t)$ increases strictly for $t$ in $[0,\tau(a)]$ and decreases strictly for $t$ in $[\tau(a),\ttt(a)]$. 

To see that there are no other solutions of (\ref{a}) for $\lambda_1\leq a<\lambda_1+\delta$ and $c\geq 0$,
we may argue by contradiction. If there were we could follow them using the parameter $c$ until $c_*(a)$.
A contradiction would result from Lemma~\ref{turn} (see the proof of Theorem~\ref{l1l2} ahead).
In alternative we could just appeal to Theorem~\ref{l1l2}.
\end{proof}

\begin{proof}[Proof of {\rm Remark~\ref{rmk3}}]
 Going back to the proof of Theorem~\ref{thmd}, reducing $\sigma$ if necessary,
the curve parametrized by $t\mapsto(a_*(t),t\phi+y_*(t),c_*(t))$ for $t\in\,]M,M+\sigma[$
may be parametrized in the form $a\mapsto(a,u_*(a),c_*(a))$, i.e.\ may be parametrized 
in terms of $a$. Under the weaker condition {\bf (b)$''$} the curve $\D_*$ may afterwards
turn back (so that it can no longer be parametrized in terms of $a$). 
But, arguing as in Lemma~\ref{L} ahead, afterwards it will stay away from $\{\lambda_1\}\times\h\times\R$.
So we may argue as in Section~\ref{pl2} to see $t\mapsto c_\phi(a,t)$ first increases and
then decreases, with no other oscillations, provided $\delta$ is chosen sufficiently small.
\end{proof}

For each fixed $a$, with $\lambda_1<a<\lambda_2$, a complete description of the solutions
of (\ref{a}) is given in Theorem~\ref{l1l2}.
\begin{proof}[Proof of {\rm Theorem~\ref{l1l2}}]
Fix $\lambda_1<a<\lambda_2$.
The Morse index of any solution $(a,u,c)$ of (\ref{a}) is less than or equal to one
and solutions with Morse index equal to one are nondegenerate. Indeed, suppose
$v\in\h$ is an eigenfunction for the linearized equation at $u$, with $L^2(\Omega)$ norm
equal to one, and $\mu$ is the corresponding eigenvalue.
Then $Q_a(v)=\mu$, where $Q_a$ is as in (\ref{q}).
The term $\int f'(u)v^2$ is nonnegative,
so if $Q_a$ was nonpositive on a two dimensional space, then $\int\left(|\nabla v|^2-av^2\right)$
would be nonpositive on that space, which is impossible, as $a<\lambda_2$.

We start at $(c,u)=(0,u_\dagger(a))$, where $u_\dagger(a)$ is the stable solution of
Section~\ref{noh}, and we start at $(c,u)=(0,0)$, a nondegenerate solution with
Morse index equal to one.
From Section~\ref{sectiond}, there exists only one degenerate solution
and that solution corresponds to a positive value of $c$.
Arguing as in the proof of Theorem~\ref{big}, we may use the Implicit Function Theorem
to follow the solutions, as $c$ increases, until they become degenerate.
From (\ref{upc}), this will have to happen, and for a positive value of $c$,
the value $c=\cc_*(a)$, where $\cc_*(a)$ is as in Theorem~\ref{thmd}.
And we may decrease $c$ from zero and follow the solutions for all negative values of $c$.

The upper bound for $c$ in inequality (\ref{upc}), and the analysis in Lemma~\ref{turn}, of the behavior
of solutions around a degenerate solution,
imply there can be no other branch of solutions besides the one
that goes through $(0,u_\dagger(a))$ and $(0,0)$. 
\end{proof}

\section{Global bifurcation at $\lambda_2$}\label{at_l2}

In addition to the previous hypotheses ($f$ satisfies {\bf (i)-(iv)} and $h$
satisfies {\bf (a)-(b)}),
henceforth we assume the domain $\Omega$ is such that {\bf ($\bm\alpha$)} holds and the
function $h$ satisfies {\bf (c)}. Recall the definition of $\beta$ in (\ref{beta}).
 
\begin{Lem}[Morse indices and nondegeneracy at $\lambda_2$]\label{oned}
 For $a=\lambda_2$ solutions of {\rm (\ref{a})} have Morse index less
than or equal to one. All solutions with Morse index equal to one are nondegenerate,
except $(\lambda_2,u,c)=(\lambda_2,t\psi,0)$ for $t\in\bigl[-\frac{M}{\beta},M\bigr]$.
\end{Lem}
\begin{proof}
Let $(\lambda_2,u,c)$ be a degenerate solution with Morse index equal to one.
The quadratic form $Q_{\lambda_2}$ in (\ref{q}) is nonpositive on a
two dimensional space ${\cal E}\subset\h$. The quadratic form 
$Q:H^1_0(\Omega)\to\R$, defined by
$Q(v):=\int\left(|\nabla v|^2-\lambda_2v^2\right)$,
is also nonpositive on ${\cal E}\subset H^1_0(\Omega)$. 

We claim ${\cal E}$
contains $\psi$. To prove the claim let $v_1$ and $v_2$ be two linearly independent 
functions in ${\cal E}$. We write 
\begin{eqnarray*}
v_1&=&\eta_1\phi+\eta_2\psi+\rrro,\\
v_2&=&\rho_1\phi+\rho_2\psi+\rrrt,
\end{eqnarray*}
with $\rrro$, $\rrrt$ orthogonal in $L^2(\Omega)$ to $\phi$ and $\psi$. If $\eta_1=0$, then
$\rrro=0$ because $Q$
is nonpositive on ${\cal E}$. In this case the claim is proved.
Similarly if $\rho_1=0$. Suppose now $\eta_1$ and $\rho_1$ are
both different from zero. The function $v:=v_2-\frac{\rho_1}{\eta_1}v_1$ is
of the form $\chi\psi+\rrr$, with $\rrr$ orthogonal to $\phi$ and $\psi$.
So $\rrr=0$. As $v$ cannot be zero, $\chi\neq 0$. The claim is proved.

As $\psi\in{\cal E}$ and $\psi^2\neq 0$ a.e.\ in $\Omega$,
it follows that $f'(u)\equiv 0$. So $u\leq M$. Thus
$$
\Delta u+\lambda_2u-ch=0.
$$
Hypothesis {\bf (c)} implies $c=0$. We obtain
$u=t\psi$ and, since $u\leq M$, $t\in\bigl[-\frac{M}{\beta},M\bigr]$.
On the other hand,
the solutions $(a,u,c)=(\lambda_2,t\psi,0)$ with $t\in\bigl[-\frac{M}{\beta},M\bigr]$
are degenerate and have Morse index equal to one.
\end{proof}
Recall the definition of $\rs$ in (\ref{s}).
\begin{Lem}[${\cal L}^{\flat\sharp}$, solutions around zero at $\lambda_2$]\label{zero}
 There exists $\delta>0$ and $C^1$ functions $\y^{\flat\sharp}:J\to{\cal S}$ and $\cfsup:J\to\R$, 
where 
$J=\bigl]-\frac{M}{\beta}-\delta,M+\delta\bigr[$, such that 
the map $t\mapsto(\lambda_2,t\psi+\y^{\flat\sharp}(t),\cfsup(t))$, defined in $J$, 
with $\y^{\flat\sharp}(t)= 0$ and $\cfsup(t)=0$ for $t\in\bigl[-\frac{M}{\beta},M\bigr]$,
parametrizes a curve ${\cal L}^{\flat\sharp}$ of solutions of {\rm (\ref{a})}. 
There exists a neighborhood of ${\cal L}^{\flat\sharp}$ 
in $\{\lambda_2\}\times\h\times\R$ such that the solutions of 
{\rm (\ref {a})} with $a=\lambda_2$ in this neighborhood lie on ${\cal L}^{\flat\sharp}$.
\end{Lem}
\begin{proof}[Sketch of the proof]
 Consider the function $\hat{g}:\R\times\rs\times\R\to L^p(\Omega)$, defined
by
$$
\hat{g}(t,y,c)=\Delta\tpsy+\lambda_2\tpsy-f\tpsy-ch.
$$
Let $t_0\in\bigl[-\frac{M}{\beta},M\bigr]$.
We know $\hat{g}(t_0,0,0)=0$. The lemma follows from the Implicit Function Theorem.
\end{proof}

\begin{Lem}[Solutions at $\lambda_2$ for $c=0$]\label{l2}
 Suppose $u$ is a solution of {\rm (\ref{b})} with $a=\lambda_2$.
Then either $u$ is a stable solution, or $u$ is a degenerate solution with
Morse index equal to one.
\end{Lem}
\begin{proof}
If $u$ is less than or equal to $M$, for then $u$ solves $\Delta u+\lambda_2 u=0$.
Thus $u=t\psi$ with $t\in\bigl[-\frac{M}{\beta},M\bigr]$ and $u$ is degenerate with
Morse index equal to one. It remains to consider the case $\max_\Omega u>M$.

 We multiply (\ref{b}) first by $u^-$ and then by $u^+$ to obtain
$$ 
\int\bigl(|\nabla u^-|^2-\lambda_2|u^-|^2\bigr)=0
$$ 
and
\begin{equation}\label{bom}
\int\bigl(|\nabla u^+|^2-\lambda_2|u^+|^2\bigr)=-\int f(u)u\leq 0.
\end{equation}
Thus
$$
\int|\nabla u^-|^2=\lambda_2\int|u^-|^2\quad
{\rm and}\quad
\int|\nabla u^+|^2\leq\lambda_2\int|u^+|^2.
$$
Since $u^-$ and $u^+$ are orthogonal in $L^2(\Omega)$ and in $H^1_0(\Omega)$,
the Dirichlet quotient,
$$
\frac{\int|\nabla v|^2}{\int v^2},
$$
is less than or equal to $\lambda_2$ on the space ${\cal E}\subset H^1_0(\Omega)$ spanned by $u^-$ and $u^+$
(with zero removed).
If $u^-\equiv 0$, then 
$u$ is a nonnegative solution of (\ref{b}) with maximum strictly greater than $M$. By Lemma~\ref{gM},
$u$ is stable.

Now we consider the case when $u^-\not\equiv 0$. Then the space ${\cal E}$ is two dimensional.
As in the proof of Lemma~\ref{oned}, ${\cal E}$ contains $\psi$.
Let $\eta_1$ and $\eta_2$ be such that
$$
\psi=\eta_1u^+-\eta_2u^-.
$$
The equality
$$
\int\bigl(|\nabla\psi|^2-\lambda_2|\psi|^2\bigr)=0
$$
implies
$$
\eta_1^2\int\bigl(|\nabla u^+|^2-\lambda_2|u^+|^2\bigr)=
-\eta_2^2\int\bigl(|\nabla u^-|^2-\lambda_2|u^-|^2\bigr)=0.
$$
Clearly $\eta_1\neq 0$, so
$$ 
 \int\bigl(|\nabla u^+|^2-\lambda_2|u^+|^2\bigr)=0.
$$ 
From (\ref{bom}),
$$
\int f(u)u=0.
$$
Therefore $u$ is less than or equal to $M$.
We have reached a contradiction. The lemma is proved.
\end{proof}

We are now in a position to give the 
\begin{proof}[Proof of {\rm Theorem~\ref{thmatl2}}]
Clearly $(\y^{\flat\sharp})'\bigl(-\frac{M}{\beta}\bigr)=0$ and $(y^{\flat\sharp})'(M)=0$.
Let $\delta$ be as in Lemma~\ref{zero}.
Reducing $\delta$ if necessary, we may assume $\max_\Omega (t\psi+\y^{\flat\sharp}(t))>M$ for $t\in\,]M,M+\delta[$, and
$\max_\Omega (t\psi+\y^{\flat\sharp}(t))>M$ for $t\in\,\bigl]-\frac{M}{\beta}-\delta,M\bigr[$.
By Lemma~\ref{oned}, and
further reducing $\delta$ if necessary to guarantee the Morse index remains equal to one, we may assume that
for $t\in\,\bigl]-\frac{M}{\beta}-\delta,M\bigr[\,\cup\,]M,M+\delta[$ the solutions
$(\lambda_2,t\psi+\y^{\flat\sharp}(t),\cfsup(t))$ are nondegenerate with Morse index equal to one.
Consequently, in a neighborhood of each one of these solutions, the curve
${\cal L}^{\flat\sharp}$ may be written in the form $(\lambda_2,u^{\flat\sharp}(c),c)$.
Therefore the restriction of $\cfsup$ to $\bigl]-\frac{M}{\beta}-\delta,M\bigr[$ must
be monotone, and the restriction of $\cfsup$ to $]M,M+\delta[$ must be monotone.

We start at $(c,u)=(0,u_\dagger(\lambda_2))$, where $u_\dagger(\lambda_2)$ is the stable solution of
Section~\ref{noh}. We may use the Implicit Function Theorem
to follow the solutions as $c$ increases, until 
$c=\cc_*(\lambda_2)$ where $\cc_*(\lambda_2)$ is as in Theorem~\ref{thmd}.
And we may decrease $c$ from zero and follow the solutions for all negative values of $c$.
The analysis in Lemma~\ref{turn}, of the behavior
of solutions around a degenerate solution,
shows that we may decrease $c$ from $\cc_*(\lambda_2)$ and follow a branch $(c,u^\sharp(c))$ of
solutions of Morse index equal to one. Lemma~\ref{oned} implies we may follow this branch 
until $c=0$. By Lemma~\ref{l2} when $c$ reaches zero the solution becomes degenerate.
When this happens Lemma~\ref{oned} says it must be of the form 
$t\psi$, with $t\in\bigl[-\frac{M}{\beta},M\bigr]$.
And Lemma~\ref{zero} says the branch must connect to the branch of solutions
described in that lemma. Without loss of generality we may assume
$\lim_{c\searrow 0} u^\sharp(c)=M\psi$. [(Otherwise we rename $\psi$ to $-\psi$).
But, in fact, we will see in the next section that this will be the case
if we choose $\psi$ so that $\int h\psi<0$].
So the restriction of $\cfsup$ to $]M,M+\delta[$ is increasing.
The restriction of $\cfsup$ to $\bigl]-\frac{M}{\beta}-\delta,M\bigr[$
must also be increasing so that $\cfsup$ is negative in this interval.
Otherwise we could increase $c$ to $c=\cc_*(\lambda_2)$ and obtain a contradiction.
Indeed, 
there would be another branch of solutions arriving
at the degenerate solution with Morse index equal to zero on the curve $\D_*$ of Section~\ref{sectiond},
and this would violate Lemma~\ref{turn}.
Since $\cfsup$ is negative in the interval $\bigl]-\frac{M}{\beta}-\delta,M\bigr[$,
Lemma~\ref{oned} implies we may use the parameter $c$ to follow this branch, which we call $(c,u^\flat(c))$,
for negative values of $c$. In fact,
with the aid of Remark~\ref{unico},
we may follow the branch $(c,u^\flat(c))$ for all negative values of $c$.
Clearly, there can be no other branches of solutions besides the one just described.
\end{proof}

\begin{proof}[Proof of {\rm Remark~\ref{rmk1}}]
Using the notations of Lemma~\ref{zero}, let
$u^{\flat\sharp}(t)=t\psi+y^{\flat\sharp}(t)$.
One easily computes 
$$
\frac{d c^{\flat\sharp}}{dt}(0)=0\qquad{\rm and}\qquad\frac{d y^{\flat\sharp}}{dt}(0)=0.
$$
By the chain rule,
$$
\frac{d u^{\flat\sharp}}{d c^{\flat\sharp}}(0)\ =\ \frac{d u^{\flat\sharp}}{d t}(0)
\frac{d t}{d c^{\flat\sharp}}(0)\ =\ \psi\times\frac{1}{0}\ =\ \infty.
$$

\end{proof}

\section{Bifurcation in a right neighborhood of $\lambda_2$}\label{pl2}

In this and the next sections, we will determine the bifurcation curves of (\ref{a}) when the value of $a$ is in a certain
right neighborhood of $\lambda_2$, which will not include $\lambda_3$.
By arguing as in Section~\ref{below_l2}, in this situation
any solution of {\rm (\ref{a})} will have 
Morse index less than or equal to two, and if it has Morse index equal to two then it
is nondegenerate.

Regarding solutions of (\ref{b}), and
analogously to Lemma~\ref{tp}, we have
\begin{Prop}[$\C_\ddagger$, solutions of {\rm (\ref{b})} bifurcating from $(\lambda_2,0)$]\label{tp2} 
Suppose $f$ satisfies {\rm\bf (i)}-{\rm\bf (iv)} and {\rm $\bm(\bm\alpha\bm)$} holds.
Suppose also $M>0$.
There exists $\delta>0$ and $C^1$ functions $a_\ddagger:J\to\R$ and $\y_\ddagger:J\to{\cal S}$,
where 
$J=\bigl]-\frac{M}{\beta}-\delta, 
M+\delta\bigr[$, such that 
the map $t\mapsto(a_\ddagger(t),t\psi+\y_\ddagger(t))$, defined in $J$, 
with $a_\ddagger(t)=\lambda_2$ and $\y_\ddagger(t)= 0$ for $t\in\bigl[-\frac{M}{\beta}, 
M\bigr]$,
parametrizes a curve $\C_{\ddagger}$ of solutions of~{\rm (\ref{b})}.
There exists a neighborhood of $\C_\ddagger\setminus\{(\lambda_2,0)\}$
in $\R\times\h$ such that the solutions of~{\rm (\ref{b})} in this neighborhood lie on $\C_\ddagger$.
\end{Prop}
The proof is similar to the one of Lemma~\ref{tp}.

We again refer to the classical paper \cite[Theorem~1.7]{CR} to assert $(\lambda_2,0)$ is
a bifurcation point.  The statement of
Proposition~\ref{tp2} also holds when $M=0$. In both cases, $M>0$ and $M=0$, in a neighborhood of 
$(\lambda_2,0)$ the solutions of~(\ref{b}) lie on $\lin\cup\C_\ddagger$.
Once more $\lin$ is the 
line parametrized by $a\mapsto (a,0)$.  

From Section~\ref{below_l2} we know the only solutions of (\ref{b})
for $\lambda_1<a<\lambda_2$ are zero and the stable solution.
From Section~\ref{at_l2} we know the only solutions of (\ref{b})
for $a=\lambda_2$ are the stable solution and $t\psi$ with $t\in \bigl[-\frac{M}{\beta},M\bigr]$.
Recalling that these (this in the case $M=0$) last solutions are (is) degenerate with Morse index equal to one,
it follows that $\A_\ddagger(t)>\lambda_2$ if $t\in\,
\bigl]-\frac{M}{\beta}-\delta,-\frac{M}{\beta}\bigr[\,\cup\,]M,M+\delta[$.

\begin{Lem}[$\D_\varsigma$, degenerate solutions with Morse index equal to one]\label{curve2} 
There exists $\sigma>0$ and $C^1$ functions $\afs:J\to\R$, $\yfs:J\to{\cal S}$, $\cfs:J\to\R$ and $\wfs:J\to
\{\zeta\in\h:\int \zeta^2=\int\psi^2\}$,
where 
$J=\bigl]-\frac{M}{\beta}-\sigma,M+\sigma\bigr[$, such that 
the map $t\mapsto(\afs(t),\ufs(t),\cfs(t))$, defined in $J$, where $\ufs(t)=t\psi+\yfs(t)$,
with $\afs(t)=\lambda_2$, $\yfs(t)= 0$, $\cfs(t)=0$,
$\wfs(t)=\psi$
for 
$t\in\,\bigl[-\frac{M}{\beta},M\bigr]$,
parametrizes a curve $\D_{\varsigma}$ of degenerate solutions of~{\rm (\ref{a})}
with Morse index equal to one.
There exists a neighborhood ${\cal O}$ of $\D_\varsigma$ 
in $\R\times\h\times\R$ such that the degenerate
solutions of the equation in ${\cal O}$ lie on $\D_\varsigma$.
The degenerate directions are given by $\wfs$.
\end{Lem}
The proof is similar to the one of Lemma~\ref{curve}.

From Section~\ref{below_l2} we know the solutions of (\ref{a})
for $\lambda_1<a<\lambda_2$ are never degenerate with Morse index equal to one.
From Section~\ref{at_l2} we know the only degenerate solutions with Morse index equal to one of (\ref{a})
for $a=\lambda_2$ are $t\psi$ with $t\in \bigl[-\frac{M}{\beta},M\bigr]$.
It follows that $$\textstyle\afs(t)>\lambda_2\ \ \ {\rm if}\ \ \ t\in\,
\bigl]-\frac{M}{\beta}-\sigma,-\frac{M}{\beta}\bigr[\,\cup\,]M,M+\sigma[.$$

\begin{Lem}[Solutions around {$\bigl\{(\lambda_2,t\psi,0):
t\in\bigl[-\frac{M}{\beta},M\bigr]\bigr\}$}]\label{V}
 There exists a neighborhood $\check{\cal V}$
 of $\bigl\{(\lambda_2,t\psi,0):t\in\bigl[-\frac{M}{\beta},M\bigr]\bigr\}$ in
$\R\times\h\times\R$, $\eps$ and $\tilde{\eps}>0$
such that solutions $(a,u,c)$ of {\rm (\ref{a})} in $\check{\cal V}$
can be parametrized in the chart
$$\cSt:=\,\,]\lambda_2-\eps,\lambda_2+\eps[\,\times
    \,\bigl]-\textstyle\frac{M}{\beta}-\tilde{\eps},M+\tilde{\eps}\bigr[.$$
by
$$(a,t)\mapsto(a,u_\psi(a,t),\cc_\psi(a,t))=(a,t\psi+\y_\psi(a,t),\cc_\psi(a,t))$$
Here $\y_\psi:\cSt\to\rs$ and $\cc_\psi:\cSt\to\R$ are $C^1$ functions.
We have $u_\psi(\lambda_2,t)=t\psi$ and $c_\psi(\lambda_2,t)=0$, for
$t\in\bigl[-\frac{M}{\beta},M\bigr]$.
\end{Lem}
The proof is similar to the argument
in the first paragraph of the proof of Theorem~\ref{global}.

We choose the value $\tilde{\eps}$ such that $\afs(t)<\lambda_2+\eps$ for 
$t\in\,\bigl]-\frac{M}{\beta}-\tilde{\eps},M+\tilde{\eps}\bigr[$.
We may assume $\check{\cal V}\subset{\cal O}$ and $\tilde{\eps}$ is smaller than $\sigma$.
Here ${\cal O}$  and $\sigma$ are as in Lemma~\ref{curve2}.
We also assume
$\eps$ and $\tilde{\eps}$ are small enough so that 
the solutions of (\ref{a}) in the closure of $\check{\cal V}$ have Morse index at least equal to one.

\begin{Lem}[No degenerate solutions with Morse index equal to one outside $\check{\cal V}$ for $c$ bounded below]\label{L}
 Let $L>0$. There exists $\hat{\delta}_L>0$ such that degenerate solutions $(a,u,c)$ of {\rm (\ref{a})}, 
with Morse index equal to one, $a<\lambda_2+\hat{\delta}_L$ and $c\geq -L$, lie in the set $\check{\cal V}$
of {\rm Lemma~\ref{V}}.
\end{Lem}
In the next section we will prove that actually one can choose $\hat{\delta}$ independently of $L$.
\begin{proof}[Proof of {\rm Lemma~\ref{L}}]
 We argue by contradiction. Suppose $(a_n,u_n,c_n)$ is a sequence of degenerate solutions of
{\rm (\ref{a})}, 
with Morse index equal to one, $a_n\leq\lambda_2+o(1)$ and $c_n\geq -L$
lying outside $\check{\cal V}$. Of course $(c_n)$
is bounded above (recall (\ref{upc})
and the solutions are bounded above by $K$ in (\ref{kkk})).  
We may assume $c_n\to c$.
From Sections~\ref{below_l2} and \ref{at_l2},
we may also assume $a_n\searrow\lambda_2$. 
From Remark~\ref{unico}, $(u_n)$ is uniformly bounded.
By \cite[Lemma~9.17]{GT} the norms
$\|u_n\|_{\h}$ are uniformly bounded. 
Thus $(u_n)$ has a strongly convergent subsequence in $L^p(\Omega)$.
Subtracting equations (\ref{a}) for $u_n$ and $u_m$ and again using \cite[Lemma~9.17]{GT},
$(u_n)$ has a subsequence which is strongly convergent, say to $u$, in $\h$.
One easily sees that $(\lambda_2,u,c)$ is a degenerate solution of (\ref{a}) with
Morse index equal to one. But $(\lambda_2,u,c)$ belongs to the closure of the complement
of $\check{\cal V}$. This contradicts Lemma~\ref{oned} and finishes the proof.
\end{proof}
We define 
\begin{equation}\label{dl}
\textstyle\delta_L=\min\bigl\{\hat{\delta}_L,\afs(M+\tilde{\eps})-\lambda_2,
\afs\bigl(-\frac{M}{\beta}-\tilde{\eps}\bigr)-\lambda_2\bigr\}
\end{equation}
(see Figure~\ref{fig7}).
Here $\afs$ is as in Lemma~\ref{curve2} and $\tilde{\eps}$ is as in Lemma~\ref{V}.
\begin{figure}
\centering
\begin{psfrags}
\psfrag{s}{{\tiny $\lambda_2-\eps$}}
\psfrag{z}{{\tiny $\lambda_2+\eps$}}
\psfrag{a}{{\tiny $a$}}
\psfrag{p}{{\tiny $\!\!\!\!\!\!\!\!(a_{\varsigma}(t),t)$}}
\psfrag{t}{{\tiny $t$}}
\psfrag{k}{{\tiny $\lambda_1$}}
\psfrag{l}{{\tiny $\lambda_2$}}
\psfrag{n}{{\tiny $M+\tilde{\eps}$}}
\psfrag{m}{{\tiny $-\frac{M}{\beta}-\tilde{\eps}$}}
\psfrag{u}{{\tiny $\lambda_2+\delta_L$}}
\includegraphics[scale=1.0]{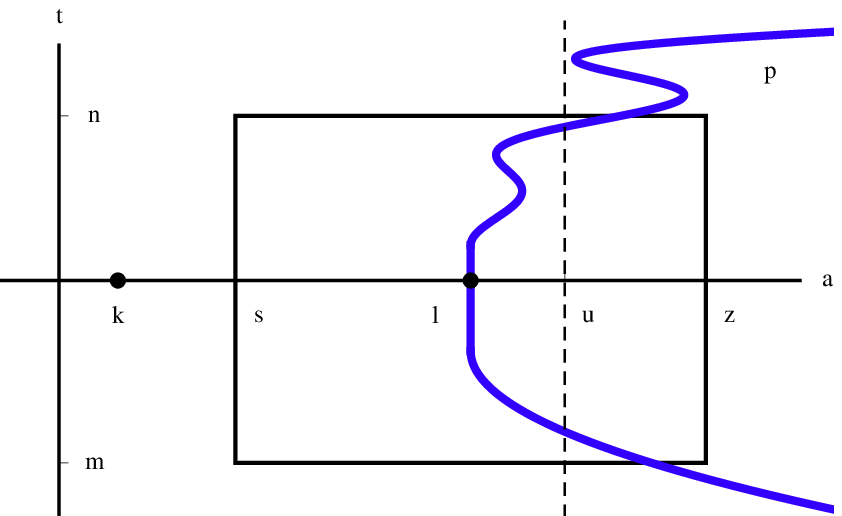}
\end{psfrags}
\caption{The chart of Lemma~\ref{V}
and the choice of the value of $\delta_L$.}\label{fig7}
\end{figure}

We will prove Theorem~\ref{thm} in the next section. For now we prove the following weaker statement. 

\begin{Prop}[Bifurcation in a right neighborhood of $\lambda_2$ for $c$ bounded below]\label{prop}
Suppose $f$ satisfies {\rm\bf (i)}-{\rm\bf (iv)}, {\rm $\bm(\bm\alpha\bm)$} holds
and $h$ satisfies {\rm\bf (a)}-{\rm\bf (c)}. Without loss of generality, suppose {\rm (\ref{hpsi})} is true.
Let $L>0$ be large and $\delta_L$ be as in {\rm (\ref{dl})}. Fix $\lambda_2<a<\lambda_2+\delta_L$.
The set of solutions of {\rm (\ref{a})} with $c>-L$ is a manifold as described in {\rm Theorem~\ref{thm}},
with $u^*$ defined in $]-L,c_*[$ and $u^\flat$ defined in $]-L,c_\flat[$.
\end{Prop} 
\begin{proof}
 We start at $(a,u,c)=(a,0,0)$. In the $(a,t)$ coordinates of Lemma~\ref{V} this solution may also be written
as $(a,0)$ as $(a,0\psi+y_\psi(a,0),c_\psi(a,0))=(a,0,0)$. Analogously to (\ref{c}),
$$ 
\frac{\partial\cc_\psi}{\partial t}(a,0)=\frac{{\textstyle\int \psi^2}}{{\textstyle\int h\psi}}\,\,(a-\lambda_2).
$$ 
This shows 
\begin{equation}\label{and}
\frac{\partial\cc_\psi}{\partial t}(a,0)>0\ {\rm for}\  a<\lambda_2
\qquad {\rm and}\qquad 
\frac{\partial\cc_\psi}{\partial t}(a,0)<0\ {\rm for}\ a>\lambda_2.
\end{equation}
Let $(a,u,c)$ be a solution with coordinates $(a,t)$.
We claim that 
if $a<\afs(t)$ then $\frac{\partial\cc_\psi}{\partial t}(a,t)>0$,
and
if $a>\afs(t)$ then $\frac{\partial\cc_\psi}{\partial t}(a,t)<0$.
Indeed, first we observe that if $a\neq\afs(t)$ then $\frac{\partial\cc_\psi}{\partial t}(a,t)\neq 0$,
because if $\frac{\partial\cc_\psi}{\partial t}(a,t)=0$, then, differentiating~(\ref{a}) with respect to $t$,
we conclude that the solution with coordinates $(a,t)$ is degenerate. But we know from Lemma~\ref{curve2} 
the only degenerate solutions of~(\ref{a}) 
in $\check{\cal V}$ lie on $\D_{\varsigma}$.
The curve parametrized by $t\mapsto(\afs(t),t)$ divides the rectangle $\cSt$
in two components, ${\cal O}_1:=\{(a,t)\in\cSt:a<\afs(t)\}$ and 
${\cal O}_2:=\{(a,t)\in\cSt:a>\afs(t)\}$. The segment $(a,0)$ with $a<\lambda_2$
is contained in the first component ${\cal O}_1$, and the segment $(a,0)$ with $a>\lambda_2$
is contained in the second component ${\cal O}_2$. Using inequalities (\ref{and}) and the continuity of $\frac{\partial\cc_\psi}{\partial t}$,
$\frac{\partial\cc_\psi}{\partial t}(a,t)>0$ in ${\cal O}_1$ and
$\frac{\partial\cc_\psi}{\partial t}(a,t)<0$ in ${\cal O}_2$. 

Since $a$ is smaller than $\lambda_2+\delta$, and $\delta$ is smaller than 
$\afs(M+\tilde{\eps})-\lambda_2$, we have $a<\afs(M+\tilde{\eps})$,
which means $(a,M+\tilde{\eps})\in{\cal O}_1$. Similarly, $a<
\afs\bigl(-\frac{M}{\beta}-\tilde{\eps}\bigr)$ so that also 
$\bigl(a,-\frac{M}{\beta}-\tilde{\eps}\bigr)\in{\cal O}_1$. Thus
\begin{eqnarray}
\frac{\partial\cc_\psi}{\partial t}(a,M+\tilde{\eps})&>&0,\label{one}\\
\frac{\partial\cc_\psi}{\partial t}\bigl(a,-\textstyle\frac{M}{\beta}-\tilde{\eps}\bigr)&>&0.\label{two}
\end{eqnarray}
As mentioned above, the parameter $a$ is fixed. We vary $t$ in $\bigl[-\frac{M}{\beta}-\tilde{\eps},M+\tilde{\eps}\bigr]$
to follow the solutions of (\ref{a}) in $\check{\cal V}$. When we arrive at the solution with coordinates
$(a,M+\tilde{\eps})$, (\ref{one}) shows we may follow the solutions out of $\check{\cal V}$ by changing coordinates,
using $c$ as a parameter and increasing it. On the opposite side, when we arrive at the solution with coordinates
$\bigl(a,-\frac{M}{\beta}-\tilde{\eps}\bigr)$, (\ref{two}) shows we may follow the solutions out of $\check{\cal V}$ by changing coordinates,
using $c$ as a parameter and decreasing it. Indeed, $c$ is increasing as one enters $\check{\cal V}$ using $t$ as an increasing parameter.

We observe that again by continuity and because the only degenerate solutions of~(\ref{a}) 
in $\check{\cal V}$ lie on $\D_{\varsigma}$, the Morse index of solutions in ${\cal O}_1$ is
equal to one and the Morse index of solutions in ${\cal O}_2$ is
equal to two. So the solutions with coordinates $(a,M+\tilde{\eps})$ and 
$\bigl(a,-\frac{M}{\beta}-\tilde{\eps}\bigr)$ are nondegenerate and have Morse index equal to one.

When we arrive at the solution with coordinates
$(a,M+\tilde{\eps})$, and follow the solutions out of $\check{\cal V}$ by 
increasing $c$,
Lemma~\ref{L} says if we find a degenerate solution it will have to have Morse index equal to zero.
Similarly, 
when we arrive at the solution with coordinates
$\bigl(a,-\frac{M}{\beta}-\tilde{\eps}\bigr)$,
and follow the solutions out of $\check{\cal V}$ by 
decreasing $c$,
Lemma~\ref{L} says if we find a degenerate solution it will have to have Morse index equal to zero.
But we know these lie on $\D_*$. So one can finish by arguing 
as in the proof of Theorem~\ref{l1l2}.

The value $\eps_\sharp$ in the statement of Theorem~\ref{thm} is the maximum value of $t\in\,]0,\tilde{\eps}[$ such that
$a=a_\varsigma(M+t)$. The value $\eps_\flat$ is the maximum value of $t\in\,]0,\tilde{\eps}[$ such that
$a=a_\varsigma\bigl(-\frac{M}{\beta}-t\bigr)$.
Also, $(c_\flat,u_\flat)=\bigl(c_\psi\bigl(a,-\frac{M}{\beta}-\eps_\flat\bigr),
u_\psi\bigl(a,-\frac{M}{\beta}-\eps_\flat\bigr)\bigr)$ and
$(c_\sharp,u_\sharp)=(c_\psi(a,M+\eps_\sharp),
u_\psi(a,M+\eps_\sharp))$.

It is clear that if $|c|$ is sufficiently small,
then {\rm (\ref{a})} has at least four solutions.
\end{proof}
\begin{proof}[Proof of {\rm Remark~\ref{rmk2}}]
This is the assertion that if $\frac{\partial\cc_\psi}{\partial t}(a,t)>0$, then $(a,t)\in{\cal O}_1$ and
so the solution with coordinates $(a,t)$ is nondegenerate with Morse index equal to one; if
$\frac{\partial\cc_\psi}{\partial t}(a,t)<0$, then $(a,t)\in{\cal O}_2$ and
so the solution with coordinates $(a,t)$ is nondegenerate with Morse index equal to two;
if
$\frac{\partial\cc_\psi}{\partial t}(a,t)=0$, then $a=a_\varsigma(t)$ and
so the solution with coordinates $(a,t)$ is degenerate with Morse index equal to one.

In alternative, at the expense of maybe another reduction in ${\cal I}_\psi$,
we could argue using the analogue of (\ref{mut}), namely
$$ 
-\mu_\psi\int {\textstyle\frac{\partial u_\psi}{\partial t}}\zeta_\psi={\textstyle\frac{\partial c_\psi}{\partial t}}\int h\zeta_\psi.
$$ 
Here $(\mu_\psi,\zeta_\psi)$ is the second eigenpair associated with linearized problem for the solution $(a,u_\psi,c_\psi)$,
obtained from the Implicit Function Theorem and with $(\mu_\psi(\lambda_2,t),\zeta_\psi(\lambda_2,t))=(0,\psi)$
for $t\in\bigl[-\frac{M}{\beta},M\bigr]$.

\end{proof}

\section{Global bifurcation in a right neighborhood of $\lambda_2$}\label{last}

In this section we prove Theorem~\ref{thm}. We assume $f$ satisfies {\rm\bf (i)}-{\rm\bf (iv)}, {\rm $\bm(\bm\alpha\bm)$} holds
and $h$ satisfies {\rm\bf (a)}, {\rm\bf (b)$'$}, {\rm\bf (c)}. The argument rests on the results above and on 
\begin{Lem}[No degenerate solutions with Morse index equal to one at infinity for $a<\lambda_2+\delta$]\label{ultimo}
There exists $\delta_0>0$ and $c_0<0$ such that degenerate solutions $(a,u,c)$ of {\rm (\ref{a})} with Morse
index equal to one and $a<\lambda_2+\delta_0$ satisfy $c\geq c_0$.
\end{Lem}
\begin{proof}
We argue by contradiction. Let $((a_n,u_n,c_n))$ be a sequence of degenerate solutions of (\ref{a}) with
Morse index equal to one,
$a_n<\lambda_2+\frac{1}{n}$ and $c_n\to-\infty$.  
We assume $c_n<0$ for all $n$.
As we saw in the proofs of Theorem~\ref{l1l2} and
Lemma~\ref{oned}, $a_n>\lambda_2$. So $a_n\to\lambda_2$.

Let $(\mu_n,w_n)$ be the first eigenpair associated with linearized problem for the solution $(a_n,u_n,c_n)$, with $w_n>0$ and
$\int w_n^2=\int\phi^2$. We have
\begin{equation}\label{f1}
\Delta w_n+a_nw_n-f'(u_n)w_n=-\mu_nw_n
\end{equation}
and $\lambda_1-a_n<\mu_n<0$.
Multiplying both sides of (\ref{f1}) by $w_n$ and integrating it follows that $(w_n)$ is bounded in $H^1_0(\Omega)$.
We may assume $w_n\weak w$ in $H^1_0(\Omega)$, $w_n\to w$ in $L^2(\Omega)$, and $w_n\to w$ a.e.\ in $\Omega$.
Clearly $w\geq 0$ and $\int w^2=\int\phi^2$.

Using (\ref{a}) and (\ref{f1}),
$$
0\leq\int(f'(u_n)u_n-f(u_n))w_n=c_n\int hw_n+\mu_n\int u_nw_n,
$$
or
$$
-c_n\int hw_n\leq -\mu_n\int u_n^-w_n\leq|\mu_n|\|u_n\|_{L^2(\Omega)}\|\phi\|_{L^2(\Omega)}.
$$
Now
$$
\int hw_n\to\int hw>0,
$$
as $h$ satisfies {\bf (b)$'$}. It follows there exists a constant $C>0$ such that
\begin{equation}\label{f2}
-\frac{c_n}{\|u_n\|_{L^2(\Omega)}}\leq C.
\end{equation}
Defining
$$
v_n=\frac{u_n}{\|u_n\|_{L^2(\Omega)}}\qquad{\rm and}\qquad d_n=\frac{c_n}{\|u_n\|_{L^2(\Omega)}},
$$
the new functions satisfy
\begin{equation}\label{f5}
\Delta v_n+a_nv_n-\frac{f(u_n)}{\|u_n\|_{L^2(\Omega)}}-d_nh=0.
\end{equation}
Inequality (\ref{f2}) shows we may assume $d_n\to d$.
\begin{Claim}\label{claim1}
There exists $C>0$ such that $f(u_n)\leq C(-c_n)$.
\end{Claim}
\begin{proof}[Proof of {\rm Claim~\ref{claim1}}]
Let $x_n$ be a point of maximum of $u_n$. As $(\Delta u_n)(x_n)\leq 0$
$$
a_nu_n(x_n)-f(u_n(x_n))-c_nh(x_n)\geq 0,
$$
or, as $u_n(x_n)>M\geq 0$
(because $a_n>\lambda_2$ and the solutions are degenerate),
\begin{equation}\label{f4}
f(u_n(x_n))\leq a_nu_n(x_n)-c_nh(x_n),
\end{equation}
\begin{equation}\label{f3}\textstyle
\frac{f(u_n(x_n))}{u_n(x_n)}\leq a_n-\frac{c_n}{u_n(x_n)}h(x_n).
\end{equation}
Suppose $u_n(x_n)\geq-c_n$ for large $n$. Then the right hand side of $(\ref{f3})$ is bounded.
Hypothesis {\bf (iv)} implies $u_n(x_n)$ is bounded and this contradicts $u_n(x_n)\geq-c_n$
as $c_n\to-\infty$. Therefore
$$
u_n(x_n)\leq-c_n\quad{\rm for\ large}\ n.
$$
Using (\ref{f4}), 
$$
f(u_n(x_n))\leq (a_n-h(x_n))(-c_n)\leq C(-c_n).
$$
As $f$ is increasing,
the claim is proved.
\end{proof}
We return to the proof of Lemma~\ref{ultimo}. Claim~\ref{claim1} and (\ref{f2}) together imply
\begin{equation}\label{f6}
\frac{f(u_n)}{\|u_n\|_{L^2(\Omega)}}\quad{\rm is\ uniformly\ bounded\ in}\ L^\infty(\Omega).
\end{equation}
We may assume
$$
\frac{f(u_n)}{\|u_n\|_{L^2(\Omega)}}\weak f_\infty\quad{\rm in}\ L^2(\Omega).
$$
Here $f_\infty\geq 0$.
Multiplying both sides of (\ref{f5}) by $v_n$ and integrating,
$(v_n)$ is bounded in $H^1_0(\Omega)$. We may assume
$v_n\weak v$ in $H^1_0(\Omega)$, $v_n\to v$ in $L^2(\Omega)$, and $v_n\to v$ a.e.\ in $\Omega$.
In fact, \cite[Lemma~9.17]{GT}
and (\ref{f5}) imply $v_n\to v$ in $C^{1,\alpha}(\bar{\Omega})$.
The function $v$ is a weak solution of 
\begin{equation}\label{f11}
\Delta v+\lambda_2v-f_\infty-dh=0.
\end{equation}

Suppose $v(x)>0$ for some $x\in\Omega$. Then, since from (\ref{f2}) the sequence $\bigl(\|u_n\|_{L^2(\Omega)}\bigr)$ is unbounded,
using {\bf (iv)},
$$
\frac{f(u_n(x))}{\|u_n\|_{L^2(\Omega)}}\ =\ \frac{f(\|u_n\|_{L^2(\Omega)}v_n(x))}{\|u_n\|_{L^2(\Omega)}v_n(x)}v_n(x)\ \longrightarrow
\ \infty\times v(x)\ =\ \infty.
$$
This contradicts (\ref{f6})
because on the set $\{x\in\Omega:v(x)>0\}$ the sequence $\Bigl(\frac{f(u_n)}{\|u_n\|_{L^2(\Omega)}}\Bigr)$
converges pointwise to $+\infty$. 
Therefore $v\leq 0$ and 
\begin{equation}\label{f30}
\sup_\Omega v_n\to 0, 
\end{equation}
as $v_n\to v$ in $C^{1,\alpha}(\bar{\Omega})$.

Let $(0,\zeta_n)$ be a second eigenpair associated with linearized problem for the solution $(a_n,u_n,c_n)$, with 
$\int \zeta_n^2=\int\psi^2$. We have
\begin{equation}\label{f20}
\Delta \zeta_n+a_n\zeta_n-f'(u_n)\zeta_n=0.
\end{equation}
\begin{Claim}\label{claim2}
There exists a linear combination of $w_n$ and $\zeta_n$ such that
$$
\xi_n:=\eta_1w_n+\eta_2\zeta_n
$$
converges strongly to $\psi$ in $H^1_0(\Omega)$.
\end{Claim}
\begin{proof}[Proof of {\rm Claim~\ref{claim2}}]
From (\ref{f1}),
$$
\int|\nabla w_n|^2\leq a_n\int w_n^2
$$
and
$$
\int|\nabla w|^2\leq \lambda_2\int w^2,
$$
where $w$ is as above. On the other hand (\ref{f20}) implies, modulo a subsequence,
$\zeta_n\weak \zeta$ in $H^1_0(\Omega)$, $\zeta_n\to \zeta$ in $L^2(\Omega)$, and $\zeta_n\to \zeta$ a.e.\ in $\Omega$ with
$$
\int|\nabla \zeta|^2\leq \lambda_2\int \zeta^2.
$$
For any linear combination $\hat{\xi}_n:=\kappa_1w_n+\kappa_2\zeta_n$ we have 
\begin{equation}\label{f9}
\Delta \hat{\xi}_n+a_n\hat{\xi}_n-f'(u_n)\hat{\xi}_n=-\mu_n\kappa_1w_n,
\end{equation}
and, as $\int w_n\zeta_n=0$, 
\begin{equation}\label{f8}
\int|\nabla \hat{\xi}_n|^2\leq a_n\int \hat{\xi}_n^2.
\end{equation}
Hence
$$
\int|\nabla (\kappa_1w+\kappa_2\zeta)|^2\leq \lambda_2\int (\kappa_1w+\kappa_2\zeta)^2.
$$
Since $\int w\zeta=0$, $w$ and $\zeta$ span a two dimensional space ${\cal E}$.
By arguing as in the proof of Lemma~\ref{oned}, there exist
$\eta_1$ and $\eta_2$ such that
$$
\psi=\eta_1w+\eta_2\zeta.
$$
We have
$$
\eta_1w_n+\eta_2\zeta_n\ \weak\ \eta_1w+\eta_2\zeta\ =\ \psi\quad{\rm in}\ H^1_0(\Omega).
$$
Taking $\kappa_1=\eta_1$ and $\kappa_2=\eta_2$ in $\hat{\xi}_n$, passing to the limit in (\ref{f8}),
and using the lower semi-continuity of the norm,
$$
\int|\nabla\psi|^2\leq\liminf\int|\nabla(\eta_1w_n+\eta_2\zeta_n)|^2\leq\lambda_2\int\psi^2.
$$
However, $\lambda_2\int\psi^2=\int|\nabla\psi|^2$ so
$$
\eta_1w_n+\eta_2\zeta_n\ \to\ \psi\quad{\rm in}\ H^1_0(\Omega).
$$
The claim is proved.
\end{proof}
We return to the proof of Lemma~\ref{ultimo}. Using (\ref{f9}),
\begin{equation}\label{f10}
\int|\nabla\xi_n|^2-a_n\int\xi_n^2+\int f'(u_n)\xi_n^2=\mu_n\eta_1^2\int w_n^2.
\end{equation}
All terms in (\ref{f10}) remain bounded as $n\to\infty$, except perhaps $\int f'(u_n)\xi_n^2$,
which must also therefore remain bounded. 
Recall $\sup_\Omega u_n>M\geq 0$.
Using $f'(u)\geq\frac{f(u)}{u}$, we may estimate this term from below as follows
\begin{eqnarray*}
\int f'(u_n)\xi_n^2&\geq&\frac{1}{\sup_\Omega v_n}\int\frac{f(u_n)}{\|u_n\|_{L^2(\Omega)}}\xi_n^2\\
&=&\frac{1}{\sup_\Omega v_n}\left[\int\frac{f(u_n)}{\|u_n\|_{L^2(\Omega)}}\psi^2+
\int\frac{f(u_n)}{\|u_n\|_{L^2(\Omega)}}(\xi_n^2-\psi^2)\right]\\
&\to&\frac{1}{0^+}\left(\int f_\infty\psi^2+0\right).
\end{eqnarray*}
We have used (\ref{f30}). This proves $f_\infty=0$ a.e.\ in $\Omega$, because by the unique continuation principle $\psi$
only vanishes on a set of measure zero.

Returning to (\ref{f11}), 
$$
\Delta v+\lambda_2v-dh=0.
$$
Hypothesis {\bf (c)} implies $d=0$. Consequently, $v$ must be a multiple of $\psi$.
But $v\leq 0$ so $v=0$, contradicting $\|v\|_{L^2(\Omega)}=1$.
We have finished the proof of Lemma~\ref{ultimo}.
\end{proof}
Our last step is the 
\begin{proof}[Proof of {\rm Theorem~\ref{thm}}]
Take $L=-c_0$, with $c_0$ as in Lemma~\ref{ultimo}, let $\hat{\delta}_{-c_0}$
be as in Lemma~\ref{L} and $\delta_{-c_0}$ be as in (\ref{dl}). 
Finally, let
$$
\delta=\min\{\delta_0,\delta_{-c_0}\}.
$$
Here $\delta_0$ is as in Lemma~\ref{ultimo}.
Degenerate solutions $(a,u,c)$ of (\ref{a}) with Morse index equal to one and $a<\lambda_2+\delta$ 
lie in the set $\check{\cal V}$ of Lemma~\ref{V}. This is a consequence of Lemmas~\ref{L}
and \ref{ultimo}. Using Remark~\ref{unico}, we may use the parameter $c$ to follow the branch of
solutions in Proposition~\ref{prop} as $c\to-\infty$.
\end{proof}
We finish with the simplest bifurcation curves one can obtain for $\lambda_2<a<\lambda_2+\delta$.
These are illustrated in Figure~\ref{fig8}.

\begin{figure}
\centering
\begin{psfrags}
\psfrag{c}{{\tiny $c$}}
\psfrag{u}{{\tiny $u$}}
\psfrag{M}{{\tiny $M\psi$}}
\psfrag{N}{{\tiny $\!\!\!\!\!\!-\frac{M}{\beta}\psi$}}
\psfrag{d}{{\tiny $\cc(\lambda_2)$}}
\psfrag{a}{{\tiny $u(a)$}}
\includegraphics[scale=.6]{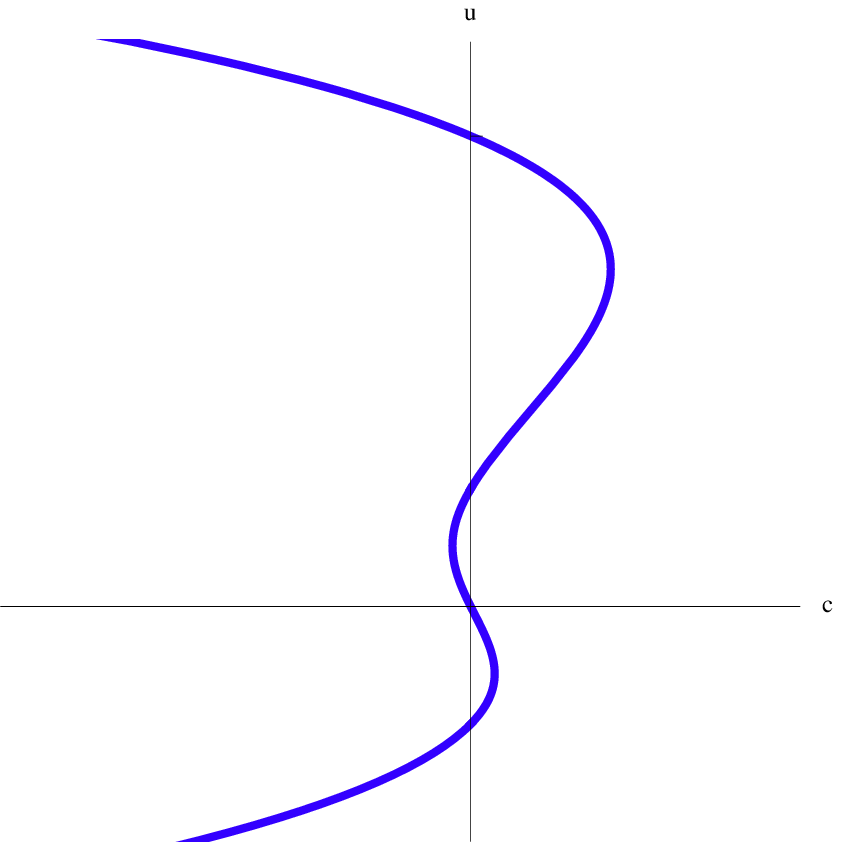}\ \ \ \ \ \ \ \ \ \ \ \ 
\includegraphics[scale=.41]{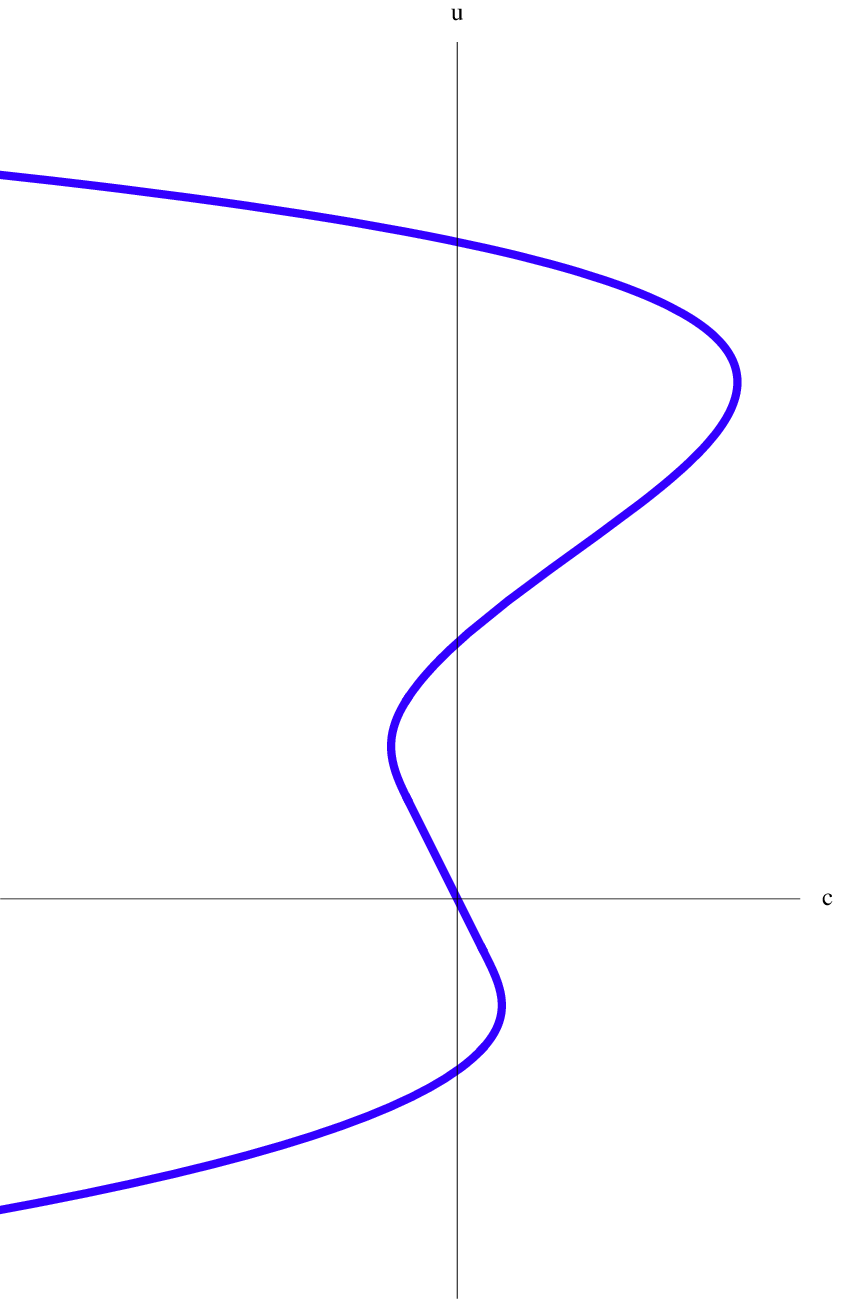}
\end{psfrags}
\caption{Bifurcation curve for $\lambda_2<a<\lambda_2+\delta$. On the left $M=0$ and on the right $M>0$.}\label{fig8}
\end{figure}

\section{Appendix}\label{appendix}

{\bf This appendix is not to appear in the final version of the paper.
It contains some proofs omitted in the text.}
\begin{proof}[Proof of {\rm Lemma~\ref{tp}}]
Consider the function $g:\R^2\times\rr\to L^p(\Omega)$,
defined by
$$
g(a,t,y)=\Delta\tpy+a\tpy-f\tpy.
$$
Let $t_0\in\,]-\infty,0[\,\cup\,]0,M]$. We know $g(\lambda_1,t_0,0)=0$.
We use the Implicit Function Theorem
to show that we may write the solutions of $g(a,t,y)=0$, in a neighborhood
of $(\lambda_1,t_0,0)$, in the form $(\A_\dagger(t),t,\y_\dagger(t))$. Let $(\alpha,z)\in\R\times\rr$.
The derivative $Dg_{(a,y)}(\alpha,z)$ at $(\lambda_1,t_0,0)$ is
\begin{eqnarray*}
g_a\alpha+g_yz 
&=&\alpha t_0\phi+\Delta z+\lambda_1z.
\end{eqnarray*}
We check that the derivative is injective. Suppose $\alpha t_0\phi+\Delta z+\lambda_1z=0$.
Multiplying both sides of this equation by $\phi$ and integrating we get $\alpha=0$.
Thus $\Delta z+\lambda_1z=0$. Since $z\in\rr$ we also get $z=0$.
This proves injectivity. It is easy to check that the derivative is also surjective.
So the derivative is a homeomorphism from $\R\times\rr$ to $L^p(\Omega)$.
\end{proof}

\begin{proof}[Proof of {\rm Theorem~\ref{big}}]
Consider the function $\tH:\h\times\R\times S\times\R\to L^p(\Omega)\times L^p(\Omega)$ ($S$ given in (\ref{ss})),
$\tH$ defined by
$$
\tH(u,c,w,\mu)=(\Delta u+au-f(u)-ch,\Delta w+aw-f'(u)w+\mu w).
$$
We may use the Implicit Function Theorem to describe the solutions of $\tH=0$
in a neighborhood of a stable solution $(u,c,w,\mu)$. 
Here $\mu$ is the first eigenvalue of the associated linearized problem and $w$ is
the corresponding positive eigenfunction on $S$.
Indeed,
\begin{eqnarray*}
\tH_uv+\tH_w\omega+\tH_\mu\nu&=&(\Delta v+av-f'(u)v,\\ &&\ \Delta \omega+a\omega-f'(u)\omega+\mu\omega-f''(u)vw+\nu w).
\end{eqnarray*}
Consider the system obtained by setting the previous derivative equal to 
$(0,0)$.
From the first equation we get $v=0$ because $\mu>0$.
Then, multiplying the second equation by $w$ and integrating by parts, we get $\nu=0$.
Thus $\omega$ is a multiple of $w$. Since $\omega$ is orthogonal to $w$, $\omega=0$.
The derivative is a homeomorphism from $\h\times\rr_w\times\R$ to $L^p(\Omega)\times L^p(\Omega)$.
So, for each fixed $a$, the solutions of $\tH=0$ in a neighborhood of a stable solution $(u,c,w,\mu)$ lie
on a $C^1$ curve parametrized by $c\mapsto(\uu^*(c),c,\w^*(c),\mu^*(c))$.
We differentiate both sides of the equations $\tH(\uu^*(c),c,\w^*(c),\mu^*(c))=(0,0)$
with respect to $c$. 
We obtain
\begin{equation}\label{sis}
\left\{\begin{array}{l}
\Delta v+av-f'(u)v=h,\\
\Delta\omega+a\omega-f'(u)\omega+\mu\omega-f''(u)vw=-\nu w,
\end{array}\right.
\end{equation}
where $v=(\uu^*)'(c)$, $\omega=(\w^*)'(c)$ and $\nu=(\mu^*)'(c)$.
When $\mu>0$ the first equation and the maximum principle give 
$$ 
(\uu^*)'(c)<0.
$$ 
From the second equation
we get 
$$ 
(\mu^*)'(c)=\frac{\int f''(u)vw^2}{\int\phi^2}<0.
$$  
By Remark~\ref{unico}, we may follow the solution $\uu^*(c)$ until it becomes degenerate. 
The solution $\uu^*(c)$ will have to become degenerate for some value of $c$.
Indeed, from (\ref{a}) we obtain
\begin{equation}\label{upc_bis}
 c\int h\phi=(a-\lambda_1)\int u\phi-\int f(u)\phi\leq (a-\lambda_1)\int u^+\phi,
\end{equation}
showing $c$ is bounded above. The solutions $\uu^*(c)$ cannot be continued for all positive values of $c$.
There must exist $c_*$ such that $\lim_{c\nearrow c_*}\mu^*(c)=0$. Clearly, the solutions $\uu^*(c)$ will converge
to a solution $u_*$ as $c\nearrow c_*$. By the uniqueness assertion in Theorem~\ref{thmd},
this $(c_*,u_*)$ belongs to $\D_*$. In particular $c_*>0$.

We have shown any branch of stable solutions can be extended for $c\in\,]-\infty,c_*[$.
But by Lemma~\ref{unique} there is a unique stable solution of (\ref{a}) for $c=0$,
namely $u_\dagger=u_\dagger(a)$. This proves uniqueness.
\end{proof}

\begin{proof}[Proof of {\rm Lemma~\ref{turn}}]
This lemma is known (\cite[Theorem~3.2]{CR2} and \cite[p.~3613]{OSS1}), but for completeness we give the proof.
Let $(c_*,u_*)$ be a degenerate solution with Morse index equal to zero.
Let 
$t_*$ and $y_*$ be such that
$u_*=t_*w_*+y_*$, with $w_*\in S$ satisfying $\Delta w_*+aw_*-f'(u_*)w_*=0$, $w_*>0$, and $y_*\in\rr_{w_*}$.
We let $\tG:\R\times\rr_{w_*}\times\R\times S\times\R\to L^p(\Omega)\times L^p(\Omega)$ be defined by
\begin{eqnarray*}
\tG(t,y,c,w,\mu)&=&(\Delta\twzy+a\twzy-f\twzy-ch,\\ &&\ \Delta w+aw-f'\twzy w+\mu w).
\end{eqnarray*}
We have $\tG(t_*,y_*,c_*,w_*,0)=0$. We may use the Implicit Function Theorem to describe the solutions of $\tG=0$
in a neighborhood of $(t_*,y_*,c_*,w_*,0)$. Indeed, at this point,
\begin{eqnarray*}
\tG_yz+\tG_c\gamma+\tG_w\omega+\tG_\mu\nu 
&=&(\Delta z+az-f'\twyz z-\gamma h,\\
&&\ \Delta\omega+a\omega-f'\twyz\omega\\ &&\ \ -f''\twyz zw_*+\nu w_*).
\end{eqnarray*}
If this derivative vanishes, then we get $\gamma=0$ and then $z=0$.
This implies $\nu=0$ and then $\omega=0$.
The derivative is a homeomorphism from $\rr_{w_*}\times\R\times\rr_{w_*}\times\R$ to $L^p(\Omega)\times L^p(\Omega)$.
So the solutions of $\tG=0$ in a neighborhood of $(t_*,y_*,c_*,w_*,0)$ lie
on a curve $t\mapsto(t,\y(t),\cc(t),\w(t),\mu(t))$. Differentiating $\tG(t,\y(t),\cc(t),\w(t),\mu(t))=(0,0)$
once with respect to $t$,
$$
\Delta z+az-f'\twyz z-\gamma h=-\left(\Delta w_*+aw_*-f'\twyz w_*\right)=0,
$$
$$
\Delta\omega+a\omega-f'\twyz\omega-f''\twyz zw_*+\nu w_*=f''\twyz w_*^2
$$
Here $z=\y'(t_*)$, $\gamma=\cc'(t_*)$, $\omega=\w'(t_*)$ and $\nu=\mu'(t_*)$. Clearly, both $\gamma$ and $z$ vanish.
This implies 
$$
\nu=\frac{\int f''\twyz w_*^3}{\int w_*^2}.
$$
It is impossible for $\max_\Omega\twyz\leq M$ because otherwise $\Delta w_*+aw_*=0$ with $w_*>0$.
Hence,
\begin{equation}\label{nu}
\mu'(t_*)>0.
\end{equation}
Differentiating the first equation in $\tG(t,\y(t),\cc(t),\w(t),\mu(t))=(0,0)$
twice with respect to $t$, at $t_*$,
\begin{eqnarray*}
\Delta z'+az'-f'\twyz z'-\gamma' h-f''\twyz z(w_*+z)=\qquad\\ 
\qquad\qquad\qquad\qquad\qquad\qquad f''\twyz w_*(w_*+z).
\end{eqnarray*}
Here $z=\y'$.
This can be rewritten as 
\begin{eqnarray*}
\Delta z'+az'-f'\twyz z'-\gamma' h&=&f''\twyz (w_*+z)^2\\ &=&f''\twyz w_*^2,
\end{eqnarray*}
as $z(t_*)=0$. Multiplying by $w_*$ and integrating,
$$
\cc''(t_*)=-\,\frac{\int f''\twyz w_*^3}{\int hw_*}.
$$
This is formula (2.7) of \cite{OSS1}.
So $\cc''(t_*)$ is negative. We recall from equation (\ref{signc}), a degenerate solution
with $a>\lambda_1$ has $c_*>0$. As $t$ increases from $t_*$,
$\cc(t)$ decreases and the solution becomes stable. 
So the ``end" of ${\cal M}^*$ coincides with the piece of curve parametrized by
$t\mapsto(c(t),t\w(t)+\y(t))$, for $t$ in a right neighborhood of $t_*$.
A parametrization of ${\bm m}^\sharp$ is obtained by taking $t$ in a left neighborhood of $t_*$.
\end{proof}

\begin{proof}[Proof of {\rm Lemma~\ref{p3}}]
 We may apply the Implicit Function Theorem as above as $(c,u)=(0,0)$ is a nondegenerate solution
of (\ref{a}) when $a$ is not an eigenvalue of the Laplacian. 
This gives the curve $\breve{\C}$ of solutions parametrized by $c\mapsto(c,\breve{u}(c))$
for $c\in\,]-\breve{c},\breve{c}[$.
To obtain more information about the behavior 
if the solutions for $a$ close to $\lambda_1$, we look at the linearized equation at $u=0$,
$$
\Delta v+av=h,
$$
where $v=\breve{u}'(0)$. As in \cite[Theorem~2, formula~(5)]{CP}, we write $v=t\phi+y$ with $y\in\rr$.
This leads to
\begin{equation}\label{l11}
v=\textstyle\frac{\int h\phi}{a-\lambda_1}\phi+(\Delta+a)^{-1}\left[h-\left(\int h\phi\right)\phi\right].
\end{equation}
Note that as $a\searrow\lambda_1$ the second term of the sum on the right hand side
approaches $(\Delta+\lambda_1)^{-1}\left[h-\left(\int h\phi\right)\phi\right]$.
We may choose a right neighborhood of $\lambda_1$ small enough so that $v>0$.
From Taylor's formula,
\begin{equation}\label{l22}
\breve{u}(c)=0+cv+o(c)\quad {\rm as}\ c\to 0. 
\end{equation}
Reducing the right neighborhood of $\lambda_1$ if necessary, we may assume $c_1<c_2$ implies $\breve{u}(c_1)<\breve{u}(c_2)$.
In particular $\breve{u}(c)$ is positive for small $c$.
Note $\breve{c}\searrow 0$ as $a\searrow\lambda_1$.
\end{proof}
\end{document}